\documentclass{psp}

\let\realverbatim=\verbatim
\let\realendverbatim=\endverbatim
\renewcommand\verbatim{\par\addvspace{6pt plus 2pt minus 1pt}\realverbatim}
\renewcommand\endverbatim{\realendverbatim\addvspace{6pt plus 2pt minus 1pt}}
\makeatletter
\newcommand\verbsize{\@setfontsize\verbsize{10}\@xiipt}
\renewcommand\verbatim@font{\verbsize\normalfont\ttfamily}
\makeatother

\title[Directions in Orbits of Geometrically Finite Hyperbolic Subgroups]
  {Directions in Orbits of Geometrically Finite Hyperbolic Subgroups}

  \author[Christopher Lutsko]{CHRISTOPHER LUTSKO\\    
  Department of Mathematics, University of Bristol,\\
  Fry Building, Bristol \textup{BS8 1TH}, United Kingdom\addressbreak
  e-mail\textup{: \texttt{chris@lutsko.com}}}

\usepackage{amssymb, amsmath, amsthm, setspace, bbm, prettyref, graphicx, float, subfigure, color, mathrsfs,mathtools}

\usepackage{Lutsko_Style}

\newtheoremstyle{named}{}{}{\itshape}{}{\bfseries}{.}{.5em}{\thmnote{#3 }}
\theoremstyle{named}
\newtheorem*{namedtheorem}{Theorem}

\newcommand{\hor}{\operatorname{hor}}
\newcommand{\BR}{\mathrm{m}^{BR}}
\newcommand{\BMS}{\mathrm{m}^{BMS}}

\begin{document}
  \maketitle
  \begin{abstract}
    We prove a theorem describing the limiting fine-scale statistics of orbits of a point in hyperbolic space under the action of a discrete subgroup. Similar results have been proved only in the lattice case with two recent infinite-volume exceptions by Zhang for Apollonian circle packings and certain Schottky groups. Our results hold for general Zariski dense, non-elementary, geometrically finite subgroups in any dimension. Unlike in the lattice case orbits of geometrically finite subgroups do not necessarily equidistribute on the whole boundary of hyperbolic space. But rather they may equidistribute on a fractal subset. Understanding the behavior of these orbits near the boundary is central to Patterson-Sullivan theory and much further work. Our theorem characterizes the higher order spatial statistics and thus  addresses a very natural question. As a motivating example our work applies to sphere packings (in any dimension) which are invariant under the action of such discrete subgroups. At the end of the paper we show how this statistical characterization can be used to prove convergence of moments and to write down the limiting formula for the two-point correlation function and nearest neighbor distribution. Moreover we establish a formula for the 2 dimensional limiting gap distribution (and cumulative gap distribution) which also applies in the lattice case.

  \end{abstract}

  \section{Introduction}

Patterson-Sullivan theory is a rich theory developed to understand the density of points in hyperbolic space near the boundary. Where the points in question make up the orbit of a fixed point under the action of a \emph{geometrically finite}\footnote{A subgroup is geometrically finite if the unit neighborhood of its convex core has finite Riemannian volume. Discrete groups whose fundamental domain is a finite-sided polygon are geometrically finite.}  (hence discrete) subgroup of the isometry group of hyperbolic space. Characterizing this density has proved tremendously fruitful as these thin groups are key players in the study of hyperbolic geometry and in many  number theoretic problems. The foundational work of Patterson \cite{Patterson1976} and Sullivan \cite{Sullivan1979} has allowed numerous authors to answer fundamental questions in number theory. These include, for example Oh and Shah's work describing the asymptotic distribution of Apollonian circle packings \cite{OhShah2012} and Bourgain, Gamburd and Sarnak's work extending classical Sieve techniques to thin groups \cite{BourgainGamburdSarnak2011}.

Patterson-Sullivan theory describes the asymptotic density of points near the boundary of hyperbolic space. Hence a very natural question one can ask is 'what about higher order spatial statistics?' For example what can one say about the gap (or nearest neighbor) distribution? Herein we will answer these questions and give a full characterization of the spatial statistics of such a point set as viewed from a fixed observer in hyperbolic space or its boundary. These questions have been addressed previously for lattices \cite{BocaPopaZaharescu2014}, \cite{KelmerKontorovich2015}, \cite{RisagerSodergren2017}, \cite{MarklofVinogradov2018}, and for certain thin groups \cite{Zhang2017}, \cite{Zhang2017b}. However we will treat a much more general class of subgroups in arbitrary dimension.

Our main results will be in general dimension $n\ge 2$. For the purpose of this introduction we restrict our attention to dimension 2 and gap statistics. The main theorem in all dimensions will follow after we present the necessary notation.

Let $G := \operatorname{PSL}(2,\R)$ and consider the left action on an element $\vect{z} \in \half^2$ via M\"{o}bius transformations

\begin{equation}
  \mat{a}{b}{c}{d}\vect{z} :=  \frac{a\vect{z}+b}{c\vect{z}+d}.
\end{equation}
Via this action $G$ is isometric to the group of orientation-preserving isometries of $\half^2$. Let $\Gamma<G $ be a \emph{Zariski dense, non-elementary, geometrically finite} subgroup and consider the orbit of a point $\vect{w}\in \half^2$, $\overline{\vect{w}} = \Gamma\vect{w}$.

For a given $t \in \R_{\ge 0}$ consider the radial projection to the boundary of all the points in $\overline{\vect{w}}$ a distance less than $t$ from $\vect{i}$. As we can identify $\partial \half^2 \cong S_1^1$ (the unit sphere) this generates a point set on $S_1^1$. Formally let $\xi(\vect{z}) \subset \half^2$ be the geodesic connecting $\vect{i}$ to $\vect{z}$ and let $\xi_s(\vect{z}) \subset \half^2$ be the point along said geodesic a distanc $s$ from $\vect{i}$ in the direction of $\vect{z}$. Define

\begin{equation} \label{eqn: P with z def 2D}
  \mathcal{Q}_{t}(\overline{\vect{w}}):= \left\{ \lim_{s \to \infty}\xi_s(\gamma\vect{w}): \gamma \in   \Gamma/\Gamma_{\vect{w}}, \; d(\gamma\vect{w},\vect{i}) < t\right\} \subset  S^1_1,
\end{equation}
where $d(\cdot,\cdot)$ denotes the hyperbolic distance and $\Gamma_{\vect{w}} := \operatorname{Stab}_{\Gamma}(\vect{w})$. Let $N_t = \# \mathcal{Q}_t(\overline{\vect{w}})$ and label the points in $\mathcal{Q}_t(\overline{\vect{w}}) $ sequentially as $\{ x_i\}_{i=1}^{N_t} \subset S^1_1$. Asymptotically the points $x_i$ will be distributed according to a so-called Patterson-Sullivan density defined below (see Subsection 2.2). That is, a consequence of {\cite[Theorem 1.2]{OhShah2013}} is that for a subset $F \subset S_1^1$

\begin{equation}\label{asymp 1}
   \# \mathcal{Q}_t(\overline{\vect{w}}) \cap F \sim C\nu_{\vect{i}}(F)e^{\delta_{\Gamma}t}
\end{equation}
where $\nu_{\vect{i}}$ is the conformal density of dimension $\delta_\Gamma$ (the Hausdorff dimension of the limit set) defined later (\ref{eqn:conf density}). The measure $\nu_{\vect{i}}$ is supported on the accumulations point of $\Gamma$ (possibly a fractal set). \eqref{asymp 1} is a consequence of \prettyref{thm:P_t Asymptotics} below.

Denote the $j^{th}$ scaled gap

\begin{equation}
  s_j:= \{x_{j+1} - x_{j}\} e^t,
\end{equation}
where $\{\cdot \}$ denotes the distance to the nearest integer and let $S(t)$ denote all the scaled gaps coming from $\cQ_t$. Define the cumulative gap distribution to be

\begin{equation}
  F_t(L):=\frac{1}{N_t}\#\{j \le N_t :s_j \ge L\}.
\end{equation}


\begin{figure}[ht!]
  \begin{center}    
    \includegraphics[width=0.9\textwidth]{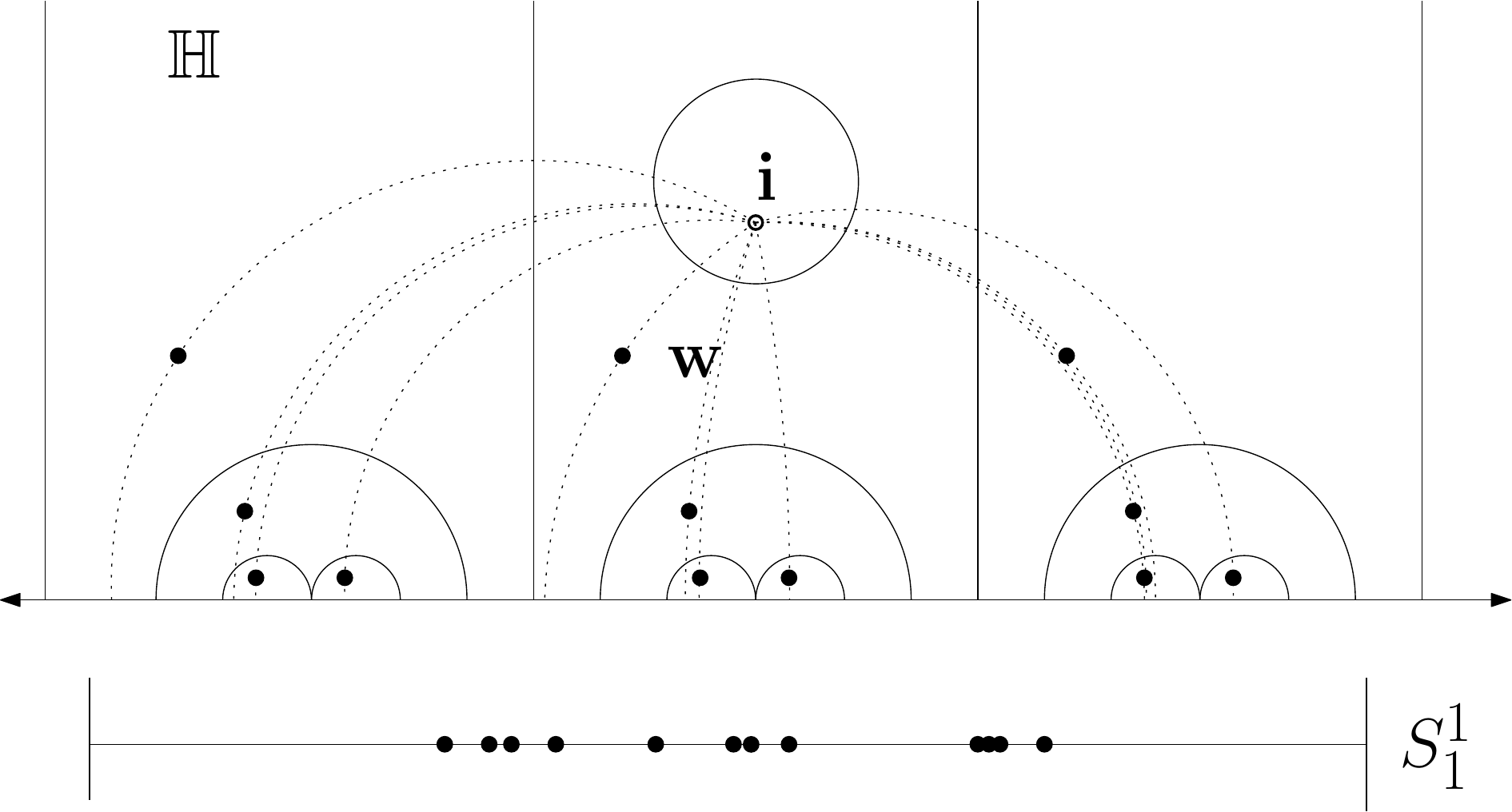}
  \end{center}
  \caption{%
    On top we show a schematic diagram of the setting in $2$ dimensions. The bold lines cut the half-plane $\half$ into fundamental domains. Then we consider a point $\vect{w} \in \half$ and the orbit $\overline{\vect{w}} = \Gamma \vect{w}$ - the black dots. The dotted lines represent the geodesics connecting the points of $\overline{\vect{w}}$ to $\vect{i}$. We consider the intersection of the geodesics with the unit hyperbolic sphere centered at $\vect{i}$ (this is equivalent to projection to the boundary $\partial \half$). Giving a projected point set on $S_1^1$ (pictured above, below the upper half-plane). If we include all points in $\overline{\vect{w}}$ such that $d(\gamma\vect{w},\vect{i})<t$ then this point set corresponds to $\mathcal{Q}_t(\overline{\vect{w}})$.   
  }%
  
  \label{fig:example}

\end{figure}


\begin{theorem} \label{thm:gap-distributions-intro}
  The limiting function $F:[0,\infty) \to \R$ defined  $F(L):=\lim_{t\to \infty} F_t(L)$ exists, is monotone decreasing and continuous.
    
Moreover we show that the gap distribution satisfies the following formula

\begin{equation}\label{gap explicit}
  F(L) = C_{\vect{w}} \int_0^{\infty} e^{\delta_{\Gamma}r}\int_0^{\pi}\prod_{\substack{\gamma \in \Gamma/\Gamma_{\vect{w}} \\ \gamma \neq \Gamma_{\vect{w}}}} \left( 1- \chi_{\mathcal{E}(\gamma)}(r,\theta)\right) d\nu_{\vect{i}}(\theta) dr,
\end{equation}
where $C_{\vect{w}}$ is an explicit constant, $\mathcal{E}(\gamma)$ is an explicit set depending on the choice of $\Gamma$, and here and throughout $\chi_{\mathcal{A}}$ is the characteristic function of the set $\mathcal{A}$.


\end{theorem}

\begin{remark}
  The proof of this Theorem will come in Section 8. This theorem generalizes a theorem by Zhang \cite{Zhang2017}. Who proved the same result hoolds for certain Schottky groups. Moreover Zhang was able to show that the cumulative gap distribution is supported away from $0$. In the general case this is not necessarily the case. Moreover, we will (in subsection 8.3) express explicitly and prove convergence of the nearest neighbor distribution in all dimensions. However this result relies on the notation developed in Section 2.
\end{remark}

In the lattice case $\delta_{\Gamma} = 1$ and $\nu_{\vect{i}}(\theta) = d\theta$. To the best of the author's knowledge \eqref{gap explicit} was not known previously. The proof of this formula is the content of Section 8.5 (where we will also take a derivative to arrive at the density). More explicit formula than this for the gap distribution are known only in the Euclidean case due to Marklof and Str\"{o}mbergsson \cite{MarklofStrom2014a} and in the hyperbolic lattice case for certain circle packing examples due to Rudnick and Zhang \cite{RudnickZhang2017}.

In this paper we will extend \prettyref{thm:gap-distributions-intro} to more general statistics and arbitrary dimension $n \ge 2$. Similar results are known only for more restricted contexts. Using number theoretic methods Boca, Popa and Zaharescu \cite{BocaPopaZaharescu2014} proved a theorem about the pair correlations of angles between directions in the \emph{modular group}. They posed a conjecture later proved by Kelmer and Kontorovich \cite{KelmerKontorovich2015} who proved a limiting distribution for the \emph{pair correlation} of angles between directions in more general \emph{hyperbolic lattices}. More recently Risager and S\"{o}dergren \cite{RisagerSodergren2017} extended these results to arbitrary dimension in the lattice case, giving effective results with explicit rates.

Marklof and Vinogradov \cite{MarklofVinogradov2018} then characterised the full limiting behaviour of such projected point sets for hyperbolic lattices. This is a special case of \prettyref{thm:main-theorem}, our main theorem, restricted to the lattice case. Zhang then proved a limiting theorem for the gap distribution of directions for \emph{certain Schottky groups} \cite{Zhang2017} (hence this was the first treatment of the infinite volume case, in 2 dimensions). Following this, Zhang proved a limiting distribution for the directions of centers of Apollonian circle packings \cite{Zhang2017b} (another non-lattice example, this time in 3 dimensions). As an application of one of our main theorems (Theorem \ref{thm:joint-limit}), in Subsection 2.4 we will discuss how our methods apply to a general class of sphere packings. That is, any sphere packing (possibly overlapping) invariant under the action of a suitable subgroup. Theorem \ref{thm:joint-limit} allows us to characterize the statistical regularity of the centers of the spheres in such a packing. 

The general strategy of this paper is the same as that used in \cite{MarklofVinogradov2018}. They use an argument of Margulis' \cite{Margulis1970} to prove equidistribution of large horospheres and spheres. Then they use those equidistribution theorems to establish the limiting distribution. Our work will follow the same plan but will instead use an equidistribution theorem  proved by Oh and Shah \cite{OhShah2013}. As the limiting measure will no longer be the invariant Haar measure there are a number of added complications.

\textbf{Plan of the paper}: In section 2 we introduce the notation and basic facts about hyperbolic geometry and the measure theory of infinite volume hyperbolic spaces necessary to state our main theorem and we present the theorem itself. At the end of section 2, as a motivating example, we will explain how our results apply to general sphere packings.

In sections 3, 4 and 5 we prove a theorem analogous to the main theorem with the observer on the boundary, $\partial \half^n$, rather than the interior, $\half^n$. Moreover we show how this limiting theorem can be used to prove convergence of the moment generating function

In sections 6 and 7 we prove our main theorem, \prettyref{thm:main-theorem} for an observer in $\half^n$.

In section 8 we present several applications: we prove the convergence of higher moments in both the boundary and interior cases, prove existence and express the limiting two-point correlation function, prove existence and express the limiting nearest neighbor distribution. Then, in dimension $n=2$, we explain how to prove \prettyref{thm:gap-distributions-intro} for gap statistics as a consequence of \prettyref{thm:main-theorem} and arrive at the explicit formula described.

  \section{Statement of Main Result}

In order to state our main result which is in general dimension $n\ge 2$ we need to introduce some of the background theory relating to higher dimensional hyperbolic geometry.

\subsection{Clifford Algebras}

For convenience we introduce the notion of Clifford numbers. This notation will be useful in describing the isometry group $G$ using an extension of complex numbers and quaternions to higher dimensions and will help with some of the calculations. What follows is a condensed introduction to the concept. For a more in-depth introduction we suggest the paper by Waterman \cite{Waterman1993}.

Define the \emph{Clifford Algebra}, $C_m$ to be the real associative algebra generated by $\vect{i}_1,...,\vect{i}_m$ such that $\vect{i}_j^2=-1$ and $\vect{i}_j\vect{i}_k = -\vect{i}_k\vect{i}_j$ for all $k \neq j$. Thus for all $\vect{a} \in C_m$

\begin{equation}
  \vect{a} = \sum_{\vect{I}} a_{\vect{I}} \vect{I}
\end{equation}
where $\vect{I}$ ranges over the products of the $\vect{i}_j$ and $a_{\vect{I}} \in \R$. $C_m$ forms a $2^m$-dimensional vector space over $\R$, which we endow with the norm $|\vect{a}|^2 = \sum_{\vect{I}}a_{\vect{I}}^2$. 

Consider the following three involutions on $C_m$

\begin{itemize}
  \item $\vect{a} \mapsto \vect{a}'$ - replaces all $\vect{i}_l$ with $-\vect{i}_l$  for all $l$
  \item $\vect{a} \mapsto \vect{a}^{\ast}$ - replaces all $\vect{I} = \vect{i}_{\nu_1},...,\vect{i}_{\nu_l}$ with $\vect{i}_{\nu_l},...,\vect{i}_{\nu_1}$
  \item $\vect{a} \mapsto \overline{\vect{a}} := \vect{a}'^{\ast}$

\end{itemize}

Define \emph{Clifford vectors} to be vectors $\vect{x} = x_0 + x_1 \vect{i}_1 + ... + x_m \vect{i}_m$ with the coresponding vector space denoted $V_m$ (which we identify with $\R^m$ in the natural way). We write $\Delta_m$ for the \emph{Clifford group}, i.e the group generated by non-zero Clifford vectors.

Furthermore we define several matrix groups

\begin{align}
  \begin{aligned}
    &\operatorname{GL}(2,C_m) := \left\{ \mat{\vect{a}}{\vect{b}}{\vect{c}}{\vect{d}} : \begin{array}{l} \vect{a},\vect{b},\vect{c},\vect{d} \in \Delta_m \cup \{0\}  \\ \vect{a}\vect{b}^{\ast},\vect{c}\vect{d}^{\ast},\vect{c}^{\ast}\vect{a},\vect{d}^{\ast}\vect{b} \in V_{m} \\ \vect{a}\vect{d}^{\ast} - \vect{b} \vect{c}^{\ast} \in \R \setminus \{0\} \end{array}   \right\},\\
    &\operatorname{SL}(2,C_m) := \left\{ \mat{\vect{a}}{\vect{b}}{\vect{c}}{\vect{d}} \in \operatorname{GL}(2,C_m) : \vect{a}\vect{d}^{\ast} - \vect{b}\vect{c}^{\ast} =1   \right\},\\
    &\operatorname{SU}(2,C_m) := \left\{ \mat{\vect{a}}{\vect{b}}{-\vect{b}'}{\vect{a}'} \in \operatorname{SL}(2,C_m)   \right\}
    \end{aligned}
\end{align}

We can then represent hyperbolic half-space by

\begin{equation}
  \half^n = \{ \vect{x} +\vect{i}y : \vect{x} \in V_{n-1}, \; y \in \R_{>0}\}
\end{equation}
with $\vect{i}:= \vect{i}_{n-1}$ (and with the usual hyperbolic metric on $\half^n$). Given a point $\vect{z} =\vect{x} +\vect{i}y \in \half^n$ we use the notation $\operatorname{Re}(\vect{z}) : = \vect{x}$ and $\operatorname{Im}(\vect{z}):= y$. Moreover the action of $\SL(2,C_m)$ on $\half^n$ defined via M\"{o}bius transformations

\begin{equation}
  \vect{z} \mapsto \mat{\vect{a}}{\vect{b}}{\vect{c}}{\vect{d}} \vect{z} = (\vect{a}\vect{z}+\vect{b})(\vect{c}\vect{z}+\vect{d})^{-1} \label{eqn:Mobius}
\end{equation}
is isometric and orientation-preserving. Therefore

\begin{equation}
  G \cong  \operatorname{PSL}(2,C_{n-1}) = \SL(2,C_{n-1}) / \{\pm 1\}
\end{equation}
is isomorphic to the group of orientation-preserving isometries of $\half^n$. The boundary of $\half^n$ can be identified

\begin{equation}
  \partial \half^n := V_{n-1} \cup \{\infty\}.
\end{equation}

Now consider the point $\vect{i}\in \half^n$, the vector $\vect{X}_{\vect{i}}\in T^{1}(\half^n)$ based at $\vect{i}$ pointed in the upwards direction, and the following relevant subgroups:

\begin{itemize}
  \item The stabilizer of $\vect{i}$ is given by 
    \begin{equation}
      K \cong \operatorname{PSU}(2,C_{n-1}) = \operatorname{SU}(2,C_{n-2}) /\{ \pm 1 \}.
    \end{equation}
    Hence we identify $\half^n \cong  G/K$.
  \item $M:= \operatorname{Stab}_G(\vect{X}_{\vect{i}})$, hence $T^1(\half^n) \cong G/M  $. Thus $M = \left\{ \begin{psmallmatrix} \vect{a} & 0 \\ 0 & (\vect{a}^{-1})^{\ast} \end{psmallmatrix} : |\vect{a}| = 1 \right\}.$
  \item $A:= \{a_r : r \in \R\}$ - one-parameter subgroup in the centralizer of $M$ such that $r\mapsto a_r$ is the unit speed geodesic flow. Thus, for a vector $\vect{u} = g\vect{X}_{\vect{i}}$ we use the notation $\vect{u}a_r = ga_r\vect{X}_{\vect{i}} = g \begin{psmallmatrix} e^{r/2} & 0 \\0 & e^{-r/2}\end{psmallmatrix}\vect{X}_{\vect{i}}$.
  \item $N_{+} := \left\{ n_+ \in G: \lim_{t\to \infty} a_{-t} n_+ a_{t} = I \right\}$ - the expanding horocycle subgroup, thus $N^{+}$ is conjugate to upper triangular matrices.
  \item $N_- := \left\{ n_- \in G: \lim_{t\to \infty} a_{t} n_- a_{-t} = I \right\}$ -  contracting horocycle subgroup (conjugate to lower triangular matrices).
\end{itemize}


\subsection{Infinite Volume Hyperbolic Spaces}

We now give an introduction to measure theory on infinite volume hyperbolic manifolds. For a more in-depth introduction in 2 dimensions we recommend the opening sections of the book by Borthwick \cite{Borthwick2007}.

For $\vect{u} \in T^1(\half^n)$ we define the geodesic endpoints in terms of the right $a_t$ action for $\vect{u}= g_{\vect{u}} \vect{X}_{\vect{i}}$
\begin{equation}
  \vect{u}^{\pm}=  \lim_{t\to \pm \infty} g_{\vect{u}}a_t \vect{X}_{\vect{i}}.
\end{equation}
Let $\delta_{\Gamma}$ denote the \emph{critical exponent of the subgroup} $\Gamma$, that is, for arbitrary $\vect{x},\vect{y} \in \half^n$

\begin{equation}
  \delta_{\Gamma} := \inf \{ s >0: \sum_{\gamma \in \Gamma} e^{-s d(\gamma \vect{x},\vect{y})}<\infty \}.
\end{equation}
Let $\Lambda(\Gamma)$ denote the \emph{limit set of} $\Gamma$ (i.e the set of accumulation points of the orbit of any point in $\half^n$, say $\vect{i}$). For the $\Gamma$ we are considering $\Lambda(\Gamma) \subset \partial \half^n$. Moreover it is well-known (\cite{Sullivan1979}) that $\delta_{\Gamma}$ is the Hausdorff dimension of $\Lambda(\Gamma)$.

For $\vect{\xi} \in \partial \half^n$ and $\vect{x},\vect{y} \in \half^n$ denote the \emph{Busemann function}, $\beta: \partial \half^n \times \half^n \times \half^n \to \R$

\begin{equation}
  \beta_{\vect{\xi}}(\vect{x},\vect{y}) = \lim_{t \to \infty} d(\vect{x},\vect{\xi}_t) - d(\vect{y}, \vect{\xi}_t)
\end{equation}
where $\vect{\vect{\xi}}_t $ lie on any geodesic ray such that as $\lim_{t\to \infty} \vect{\xi}_t = \vect{\xi}$ (the limiting value is independent of the choice of ray). In words $\beta_{\vect{\xi}}(\vect{x},\vect{y})$ is the signed geodesic distance between two horospheres each based at $\vect{\xi}$ containing $\vect{x}$ and $\vect{y}$ respectively.

With that, let $\{\mu_{\vect{x}} : \vect{x} \in \half^n\}$ denote a family of measures on $\partial \half^n$. We call such a family a \emph{$\Gamma$-invariant conformal density of dimension $\delta_{\mu} >0$} if: for each $\vect{x} \in \half^n$, $\mu_{\vect{x}}$ is a finite Borel measure such that 
\begin{align}
  \begin{aligned}\label{eqn:conf density}
  &\gamma_{\ast} \mu_{\vect{x}} ( \cdot) := \mu_{\vect{x}}(\gamma^{-1}( \cdot )) = \mu_{\gamma\vect{x}}(\cdot)\\
  &\frac{d\mu_{\vect{x}}}{d\mu_{\vect{y}}}(\vect{\xi}) = e^{\delta_\mu \beta_{\vect{\xi}}(\vect{y},\vect{x})},
  \end{aligned}
\end{align}
for all $\vect{y} \in \half^n$, $\vect{\xi} \in \partial \half^n$, and $\gamma \in \Gamma$.

Patterson in dimension $2$ \cite{Patterson1976} and Sullivan \cite{Sullivan1979} for general dimension, proved the existence of a $\Gamma$-invariant conformal density of dimension $\delta_{\Gamma}$, the critical exponent, supported on $\Lambda(\Gamma)$ which we will denote $\{\nu_{\vect{x}} : \vect{x} \in \half^n\}$ - the \emph{Patterson-Sullivan density}. Moreover let the \emph{Lesbegue density}, $\{\mathrm{m}_{\vect{x}} : \vect{x} \in \half^n\}$ denote the $G$-invariant conformal density of dimension $(n-1)$, unique up to homothety. 

From here we can define several measures on $T^{1}(\half^n)$ which will be essential to what follows. For $\vect{u} \in T^{1}(\half^n)$, let $\pi(\vect{u})$ be the projection to $\half^n$, $s:= \beta_{\vect{u}^-}(i,\pi(\vect{u}))$ and define

\begin{itemize}
  \item The \emph{Bowen-Margulis-Sullivan} measure, given by
    \begin{equation}
      d\mathrm{m}^{BMS}(\vect{u}) = e^{\delta_{\Gamma}\beta_{\vect{u}^+}(\vect{i},\pi(\vect{u}))}e^{\delta_{\Gamma}\beta_{\vect{u}^-}(\vect{i},\pi(\vect{u}))}d\nu_{\vect{i}}(\vect{u}^+)d\nu_i(\vect{u}^-)ds.
    \end{equation}
    This measure is supported on $\{\vect{u} \in T^{1}(\half^n) : \vect{u}^+,\vect{u}^- \in \Lambda(\Gamma)\}$ and is finite on $T^1(\Gamma \backslash \half^n)$ for geometrically finite $\Gamma$ \cite{Sullivan1979}.
  \item The \emph{Burger-Roblin} measure
    \begin{equation}
      d\mathrm{m}^{BR}(\vect{u}) = e^{\delta_{\Gamma}\beta_{\vect{u}^-}(\vect{i},\pi(\vect{u}))}e^{(n-1)\beta_{\vect{u}^+}(\vect{i},\pi(\vect{u}))}d\nu_{\vect{i}}(\vect{u}^-)d\mathrm{m}_{\vect{i}}(\vect{u}^+)ds.
    \end{equation}
    This measure is supported on $\{\vect{u} \in T^{1}(\half^n) : \vect{u}^- \in \Lambda(\Gamma)\}$ and is, in general, not finite on $T^1(\Gamma \backslash \half^n)$.

\end{itemize}

These are both measures on $T^{1}(\half^n) \cong G/M$. We extend them to measures on $G$. That is, let $\mu$ be either $m^{BR}$ or $m^{BMS}$ defined on $T^1(\half^n)$, for $\phi \in \mathcal{C}_{c}(G)$

\begin{equation}
  \int_{G} \phi (g) d\mu(g) = \int_{T^1(\half^n)} \int_{M} \phi(\vect{u}m) d\mu^{Haar}_{M}(m) d\mu(\vect{u})
\end{equation}
where $\mu^{Haar}_{M}(m)$ is the normalized probability Haar measure on $M$. Thus we simply average out the extra dependence. To avoid too much notation we denote the $BR$-measures on $G$ and $T^1(\half^n)$ both by $\mathrm{m}^{BR}$ and likewise for the $BMS$-measure.

Furthermore, let $H<G$ be an expanding horosperical subgroup for the \emph{right} $a_r$-action (i.e a subgroup of $N_-$). Let $\overline{H}:= H/(M\cap H)$ be the projection to $T^1(\half^n)$. For a fixed $g \in G$, define

\begin{equation} \label{PS def}
  d\mu_{g\overline{H}}^{PS}(gh):= e^{\delta_{\Gamma}\beta_{gh\vect{X}_{\vect{i}}^+}(\vect{i},gh\vect{i})} d\nu_{\vect{i}}(gh\vect{X}_{\vect{i}}^+).
\end{equation}
Given a horospherical subgroup $H$, $\overline{H}$ is isomorphic with a horosphere in $T^1(\half^n)$. Hence there exists a group isomorphism 

\begin{equation}
  \hor: \R^{n-1} \to \overline{H}   \label{eqn:hor meas}
\end{equation}
such that the push-foward of the Haar measure is equal to the Lebesgue measure

\begin{equation}
  d\mu_{\overline{H}}^{Haar}(\hor^{-1}(\vect{x})) = d\vect{x}.
\end{equation}
Define the measure on $\R^{n-1}$

\begin{equation}\label{omega def}
  d\omega_{\Gamma,g,\overline{H}}^{PS}(\vect{x}) := d\mu_{\Gamma g \overline{H}}^{PS}(g\hor^{-1}(\vect{x})).
\end{equation}

Finally, let $\overline{K} = K/M$ and define the \emph{spherical Patterson-Sullivan} measure to be

\begin{align}\label{spherical PS}
  d\mu^{PS}_{\Gamma g \overline{K}}(gk) := e^{\delta_{\Gamma}\beta_{gk\vect{X}_{\vect{i}}}(\vect{i},gk a_{1}\vect{i})}d\nu_{\vect{i}}(gk\vect{X}_{\vect{i}}).
\end{align}
For a fixed $g \in G$, the prefactor $e^{\delta_{\Gamma} \beta_{gk\vect{X}_{\vect{i}^+}}(\vect{i}, g k a_1 \vect{i})}$ is constant.

Unlike for horospheres there is not a single natural way to parameterise spheres. Therefore we add a Jacobian to ensure the parameterised Patterson-Sullivan measure is invariant for different parameterisations. Specifically we use the following polar coordinate change of variables.

\begin{lemma}\label{lem:polar}
  For $k \in \overline{K}$ let $u = k^{-1} \vect{0}$. Writing $k = \begin{psmallmatrix} \vect{a} & \vect{b} \\ - \vect{b}' & \vect{a}' \end{psmallmatrix}$ we have the following change of variables
  \begin{equation}\label{polar}
    du = |\vect{a}|^{n-1} dk.
  \end{equation}
\end{lemma}

\begin{proof}\phantom{\qedhere}
  While this is classical we present a proof using conformal densities for completeness. First note

  \begin{equation}
    du = e^{(n-1)\beta_{n_+(-u) \vect{X}_{\vect{i}}^-}(\vect{i},n_+(-u)\vect{i})} dm_{\vect{i}}(n_+(-u)\vect{X}_{\vect{i}}).
  \end{equation}
  Since $u = - \vect{b}\vect{a}^{-1}$ we can write

  \begin{equation}
    (\transpose{k})^{-1} = - n(-u)\begin{pmatrix} |\vect{a}^{-1}| & 0 \\ 0 & |\vect{a}| \end{pmatrix} \begin{pmatrix} \frac{\vect{a}^\prime + \vect{b}^\prime \vect{a}^{-1} \vect{b}}{|\vect{a}|^{-1}} & 0 \\ 0 & \frac{-\vect{a}}{|\vect{a}|} \end{pmatrix} \begin{pmatrix} 1 & 0 \\ -\vect{a}^{-1}\vect{b} & 1 \end{pmatrix}
  \end{equation}
  where $\transpose{k}$ denotes the transpose. Note that the rightmost matrix is in $N_-$, the second from the right is in $M$ and the third is in $A$. Therefore
  \begin{equation}
    n_+(-u)\vect{X}_{\vect{i}}^- = \transpose{k}^{-1}\vect{X}^-_{\vect{i}}.
  \end{equation}
  Moreover

  \begin{equation}
    \beta_{n_+(-u) \vect{X}_{\vect{i}}^-}(\vect{i},n_+(-u)\vect{i}) = \ln|\vect{a}| + \beta_{(\transpose{k})^{-1} \vect{X}_{\vect{i}}^-}(\vect{i},(\transpose{k})^{-1}\vect{i}) = \ln|\vect{a}| + \beta_{(\transpose{k})^{-1} \vect{X}_{\vect{i}}^-}(\vect{i},\vect{i}) = \ln|\vect{a}|.
  \end{equation}
  Thus

  \begin{equation}
    du = |\vect{a}|^{n-1} dm_{\vect{i}}((\transpose{k})^{-1}\vect{X}_{\vect{i}}).
  \end{equation}
  The measure $dm_{\vect{i}}((\transpose{k})^{-1}\vect{X}_{\vect{i}}) = d(\transpose{k})^{-1} = dk$. Proving Lemma \ref{lem:polar}.

\end{proof}

Now fix $g \in G$ and a parameterisation $\vect{x} \mapsto R(\vect{x}) \in \overline{K}$ with $\vect{x}$ ranging in a non-empty open set $\mathcal{U} \subset \R^{n-1}$. Let $\tilde{\vect{x}} = R(\vect{x}) \vect{0}$ and $\left| \frac{\partial \tilde{\vect{x}}}{\partial \vect{x}}\right|$ the standard Jacobian on $\R^{n-1}$. Define the \emph{parameterised spherical Patterson-Sullivan measure} for $\mathcal{U}$ to be 

\begin{equation}\label{omega PS K def}
  d \omega^{PS}_{\Gamma,g,\overline{K}}(\vect{x}) = \left| \frac{\partial \tilde{\vect{x}}} {\partial \vect{x}}\right|^{-1} |\vect{a}|^{n-1} d\mu^{PS}_{\Gamma g \overline{K}}(gR(\vect{x})).
\end{equation}

\subsection{Main Theorem}

Given two points $\vect{w}, \vect{z} \in \half^n$ define the \emph{direction function}, $\varphi_{\vect{z}}(\vect{w})$, to be the intersection of the geodesic connecting $\vect{z}$ to $\vect{w}$ with the hyperbolic unit sphere centered at $\vect{z}$ (i.e $\overline{K}e^{-1}\vect{i} + \vect{z}$). Thus $\varphi: \half^n \times \half^n \to S_1^{n-1}$.

Fix $\Gamma<G$ a \emph{Zariski dense, non-elementary, geometrically finite} subgroup. Given the orbit $\overline{\vect{w}} = \Gamma \vect{w}$ and $s < t \in \R_{\ge 0}$ define

\begin{equation} \label{eqn: P with z def}
  \mathcal{P}_{t,s}^{\vect{z}}(\overline{\vect{w}}):= \{ \varphi_{\vect{z}}(\gamma\vect{w}): \gamma \in   \Gamma/\Gamma_{\vect{w}}, \; s < d(\gamma \vect{w},\vect{z}) < t\},
\end{equation}
Thus $\mathcal{P}_{t,s}^{\vect{z}}(\overline{\vect{w}})$ represents the set of directions of orbit points of $\vect{w}$ within an annulus (of inner radius $s$ and outer radius $r$) around the observer at $\vect{z}$.

Without loss of generality we can use the left-invariance of the metric $d$ to move $\vect{w}$ and set $\vect{z}$ to be $\vect{i}$ (replacing $\Gamma$ with a new subgroup of $G$, conjugate to $\Gamma$). Set

\begin{equation} \label{eqn: P with i def}
  \mathcal{P}_{t,s}(\overline{\vect{w}}):= \mathcal{P}_{t,s}^{\vect{i}}(\overline{\vect{w}}).
\end{equation}
The first order statistics of this projected point set are characterized by a result of Oh and Shah \cite{OhShah2013}


\begin{theorem} \label{thm:P_t Asymptotics}
  Let $F \subset \overline{K} \cong S^{n-1}_1$ with $\nu_{\vect{i}}(\partial F) =0$. Then the following asymptotic formula holds as $t \to \infty$

  \begin{equation}
    \#(\mathcal{P}_{t,0}(\overline{\vect{w}})\cap F) \sim \frac{|\mu^{PS}_{\Gamma \overline{K}}|}{\delta_{\Gamma}|m^{BMS}|} \nu_{\vect{i}}(F)e^{\delta_{\Gamma}t}
  \end{equation}
\end{theorem}
\noindent This theorem follows from {\cite[Theorem 7.16]{OhShah2013}}. 

Turning now to our main object of study: the higher order spatial statistics. Let $\omega$ denote the solid angle measure on $S_1^{n-1}$ normalized to be a probability measure. Hence, for a subset $\mathcal{A} \subset S^{n-1}_1$,

\begin{equation} \label{eqn: omega def}
  \omega(\mathcal{A}) = \frac{ \vol_{S^{n-1}_1}(\mathcal{A})}{\vol_{S^{n-1}_1}(S_1^{n-1})}.
\end{equation}
For $\sigma>0$ let $\mathcal{D}_{t,s}(\sigma, \vect{v}, g\overline{\vect{w}}) \subset S_1^{n-1}$ be the (shrinking with $t$) open disk centered at $\vect{v} \in S^{n-1}_1$ of volume

\begin{equation} \label{eqn: D volume}
  \omega(\mathcal{D}_{t,s}(\sigma, \vect{v}, g\overline{\vect{w}})) =  \frac{\sigma}{\# \mathcal{P}_{t,s}(g\overline{\vect{w}})^{\frac{n-1}{\delta_{\Gamma}}}},
\end{equation}
the scaling in the exponent is chosen in such a way that $\mathcal{D}$ scales like in the lattice-case (we will discuss this scaling after the statement of \prettyref{thm:main-theorem}). Let

\begin{equation}
  \mathcal{N}_{t,s}(\sigma, \vect{v}, g\overline{\vect{w}}):= \#(\mathcal{P}_{t,s}(g\overline{\vect{w}}) \cap \mathcal{D}_{t,s}(\sigma,\vect{v}, g\overline{\vect{w}})).
\end{equation}

Finally define the cuspidal cone:

\begin{equation}\label{cusp cone}
  \mathcal{Z}_0(s,\sigma) := \{\vect{z} \in \half^n : \operatorname{Re}(\vect{z}) \in \vartheta^{-1/\delta_{\Gamma}}\mathcal{B}_{\sigma}, \; 1\le \operatorname{Im}(z) \le e^s \},
\end{equation}
where $\vartheta = \frac{|\nu_{\vect{i}}|}{\delta_{\Gamma} \BMS}$ and $\mathcal{B}_{\sigma}$ is a ball (in $\R^{n-1}$) of volume $\sigma$ centered at the origin. With that, the main theorem is:


\begin{theorem} \label{thm:main-theorem}
  Let $\lambda$ be a Borel probability measure on $S_1^{n-1}$ with continuous density with respect to Lesbegue. Then for every $g\in G$, $r \in \Z_{>0 }$, $s \in [0,\infty]$ and $\sigma \in (0,\infty)$
  
  \begin{equation} 
    E_s(r,\sigma;g\overline{\vect{w}}):= \lim_{t\to \infty} e^{(n-1-\delta_{\Gamma})t}\lambda (\{\vect{v} \in S_1^{n-1}: \mathcal{N}_{t,s}(\sigma, \vect{v}; g\overline{\vect{w}}) = r\})
  \end{equation}
  exists and is given by:
  
  \begin{equation} \label{eqn:main-limit-eqn}
    E_s(r,\sigma; g\overline{\vect{w}}) = \frac{C_{\lambda}}{|\mathrm{m}^{BMS}|}\mathrm{m}^{BR}(\{ \alpha \in \Gamma \backslash G : \#(\alpha^{-1}\overline{\vect{w}} \cap \mathcal{Z}_0(s,\sigma))=r\})
  \end{equation}
  where $C_{\lambda}= C_{\lambda}(g,\Gamma) = \int_{\overline{K}}\lambda'(k)d\mu^{PS}_{\Gamma g \overline{K}}$. Moreover the limit distribution $E_s(\cdot,\sigma; g\overline{\vect{w}})$ is continuous in $s\in (0,\infty]$ and $\sigma \in (0,\infty)$ and satisfies:

  \begin{equation} \label{eqn:sigma-limit}
    \lim_{\sigma \to 0} E_s(r,\sigma, \overline{\vect{w}}) =  0
  \end{equation}
\end{theorem}

\begin{remark}
  In section 8 we will show several consequences of the above theorem. Namely we show how to prove convergence of moments and prove existence and write explicitly the two-point correlation and gap statistics.
\end{remark}

\begin{remark}
  The above theorem is not true in general for $r=0$, unlike the case for lattices. When considering lattices, Marklof and Vinogradov also have a theorem of the same form with $r\ge 0$. The reason for this discrepency is that the scaling of the set $\mathcal{D}_{t,s}(\sigma, \vect{v}, g\overline{\vect{w}})$, (\ref{eqn: D volume}) is the same scaling as one would expect for lattices. Hence, when we consider orbit-point-free sets the scaling factor $e^{(n-1-\delta_{\Gamma})t}$ is too large and causes the integral to blow up. In other words, there are two scales to this problem. For the two dimensional problem this transates to the fact that most gaps between neighboring directions are of size $e^{-(n-1)t}$ but there are very big gaps of size $e^{-\delta_{\Gamma} t}$. This dichotomy was pointed out by Zhang \cite{Zhang2017}. 
\end{remark}

\subsection{Sphere Packings}

In section 4 we will replicate \prettyref{thm:main-theorem}, with the observer moved to $\vect{\infty}$ and rather than consider a ball centered at the observer, we will consider an expanding horosphere based at the point $\vect{\infty}$. This will induce a similar point set to \eqref{eqn: P with z def} which we will denote $\mathcal{P}^{\vect{\infty}}_{t,s}(\overline{\vect{w}})$. In which case \prettyref{thm:joint-limit} below, implies the analogous result as \prettyref{thm:main-theorem} for this point set. Using that, we can describe the spatial regularity of general sphere packings. For a general discussion of such packings see {\cite[Section 7]{Oh2014}}. We include here a brief discussion of this application as a motivating example.

For $n \ge 3$, by a sphere packing, we mean the union of a collection of (possibly intersecting) $(n-2)$-spheres. Let $\mathcal{P}$ be a sphere packing in $\R^{n-1}$ invariant under the right action of a Zariski dense, non-elementary, geometrically finite subgroup. When $n=3$ the canonical example of such a sphere packing is the Apollonian circle packing, however many other examples exist. Another nice example is considered in \cite{Kontorovich2017}, wherein Kontorovich considers so-called Soddy packings which generalize the Apollonian case to dimension $n=4$ (our discussion here holds for more general packings as well).

A natural problem is to understand the asymptotic characteristics of such a collection as one restricts the set of spheres to those of radius larger than a certain cut off. Asymptotic counting formula for these packings are given in {\cite[Theorem 7.5]{Oh2014}}.  And, in the Apollonian case for $n=3$, \cite{Zhang2017b} studied the spatial statistics of the centers of these packings. In fact, a special case of \prettyref{thm:joint-limit} (below) characterizes the spatial statisics of these packings. To see this, we simply point out a well known relationship. 

Let $\mathcal{P}$ be a $\Gamma$-invariant sphere packing in $\R^{n-1} \cong \partial \half^n$. Now let $\tilde{\mathcal{P}}$ be the collection of hemispheres supported on $\mathcal{P}$ (i.e whose intersection with $\partial \half^n$ is $\mathcal{P}$). In this case $\tilde{\mathcal{P}}$ is also $\Gamma$ invariant. 

Let $\vect{w} \in \half^n$ denote the apex of one of the spheres in $\tilde{\mathcal{P}}$. Then $\overline{\vect{w}} = \Gamma \vect{w} $ denotes the collection of apices of the spheres in $\tilde{\mathcal{P}}$. Hence, using the notation of section 4, the set

\begin{equation}
  \mathcal{P}_{t,s}^{\vect{\infty}}(\overline{\vect{w}}) := \{ \operatorname{Re}(\gamma\vect{w}): \gamma \in  \Gamma_{\vect{\infty}} \backslash \Gamma /\Gamma_w, \; e^{-t} \le \operatorname{Im}(\gamma \vect{w}) < e^{s-t} \},
\end{equation}
is equivalent to

\begin{equation}
  \mathcal{P}_{t,s}^{\vect{\infty}}(\overline{\vect{w}}) := \{ c(S): S \in \mathcal{P} , \;  e^{-t} \le r(S) < e^{s-t} \},
\end{equation}
where $c(S)$ is the location of the center of the sphere $S \in \mathcal{P}$ and $r(S)$ is the radius of $S$. In particular $\mathcal{P}_{t,\infty}^{\vect{\infty}}(\overline{\vect{w}})$ denotes the centers of all of the spheres with radius larger than $e^{-t}$. Hence \prettyref{thm:joint-limit} describes the asymptotic spatial characteristics of this point set for any sphere packing (invariant under the action of non-elementary, Zariski dense subgroups).

  \section{Equidistribution Theorems}

Our goal is to apply an equidistribution theorem of Oh and Shah {\cite[Theorem 3.6]{OhShah2013}}. However their theorem applies only to $M$-invariant functions whereas we need an equidistribution theorem for functions on $G$. A similar equidistribution theorem for functions of $G$ was proved by Mohammadi and Oh {\cite[Theorem 5.3]{MohammadiOh2015}} - however they use spectral methods and hence assume a lower bound on the critical exponent (thus giving them an exponential rate).

Fortunately the exact proof of {\cite[Theorem 3.6]{OhShah2013}} can be used to prove the necessary theorem (without the exponential rate). Let $H$ be an unstable horospherical subgroup for \emph{right} multiplication by $a_t$, therefore $H < N_-$. 


\begin{theorem}   \label{thm:Mohammadi-Oh}
  For any $g \in G$, any $\Psi \in \mathcal{C}_c(\Gamma\backslash G)$ and $\phi \in \mathcal{C}_c(gH)$

  \begin{multline}\label{eqn: MO15 equi}
    \lim_{t \to \infty} e^{(n-1-\delta_{\Gamma})t} \int_{\overline{H}}\int_{H \cap M} \Psi(\Gamma g hm a_t) \phi(ghm) d\mu^{Haar}_{\overline{H}}(h)d\mu^{Haar}_{H \cap M}(m)\\  = \frac{1}{|\mathrm{m}^{BMS}|}\int_{H \times \Gamma \backslash G}\Psi(\alpha) \phi(gh) d\mathrm{m}^{BR}(\alpha)d\mu_{\Gamma gH}^{PS}(gh).
  \end{multline}

\end{theorem}
The proof of this theorem is omitted as it is identical to the proof of {\cite[Theorem 3.6]{OhShah2013}} with one exception: rather than use the mixing theorem of Rudolph, Roblin and Babillot on $T^1(\Gamma \backslash \half^n)$, (which appears as {\cite[Theorem 3.2]{OhShah2013}}) use a mixing theorem for the $BMS$ measure under the frame flow on $G$ proved by Winter {\cite[Theorem 1.1]{Winter2015}}. Namely, write $g\in G$ as $g = \vect{u} m$ for $\vect{u}\in T^1(\half)$ and $m \in M$. From there, using Winter's mixing theorem and the fact that the frame flow is in the centralizer of $M$, the same proof will give the above theorem.

We can now replace \prettyref{thm:Mohammadi-Oh} with the following corollary


\begin{corollary} \label{cor:prob-density}
  Under the assumptions of \prettyref{thm:Mohammadi-Oh}, let $\lambda$ be a Borel probability measure on $\R^{n-1}$ with density $\lambda' \in \mathcal{C}_c(\R^{n-1})$. Then for any $g \in G$

  \begin{multline}\label{cor 3.2}
    \lim_{t\to \infty} e^{(n-1-\delta_{\Gamma})t}\int_{\R^{n-1}}\int_{M\cap H} \Psi(\Gamma g\hor (\vect{x})m a_t) d \lambda(\vect{x})d\mu^{Haar}_{H \cap M}(m) \\= \frac{1}{|\mathrm{m}^{BMS}|}\int_{ \R^{n-1} \times\Gamma \backslash G}\lambda'(\vect{x})\Psi(\alpha)d\mathrm{m}^{BR}(\alpha)d\omega_{\Gamma, g, \overline{H}}^{PS}(\vect{x}).
  \end{multline}
\end{corollary}

\begin{proof}\phantom{\qedhere}

Inserting the definition of $\lambda'$ and then applying Theorem \ref{thm:Mohammadi-Oh} with $\phi(\cdot) = \lambda^\prime \circ \hor^{-1}( g^{-1} (\cdot) M)$ gives

\begin{align}
  \begin{aligned}\notag
   &\lim_{t\to \infty} e^{(n-1-\delta_{\Gamma})t}\int_{\R^{n-1}}\int_{M\cap H} \Psi(\Gamma g\hor (\vect{x})m a_t) d\mu^{Haar}_{H \cap M}(m) d \lambda(\vect{x}) \\
  &\phantom{=}= \lim_{t\to \infty} e^{(n-1-\delta_{\Gamma})t}\int_{\overline{H}}\int_{H\cap M}\Psi(\Gamma g h m a_t) \lambda'(\hor^{-1}(g^{-1}(ghm)M))d\mu^{Haar}_{H\cap M}(m)d\mu_{\overline{H}}^{Haar}(h) \\
  &\phantom{=}= \frac{1}{|m^{BMS}|} \int_{\overline{H} \times \Gamma \backslash G } \Psi(\alpha)\lambda'(\hor^{-1}(h)) d\mathrm{m}^{BR}(\alpha)d\mu^{PS}_{\Gamma g \overline{H}}(h)
  \end{aligned}
\end{align}
Now inserting the parameterisation $\hor^{-1}:\overline{H} \to \R^{n-1}$ gives \eqref{cor 3.2}

\end{proof} 

From here, the proof of {\cite[Theorem 5.3]{MarklofStrom2010}} allows us to extend to functions of $\R^{d-1} \times \Gamma \backslash G$ and to sequences of functions


\begin{theorem} \label{thm:multiple-functions}
  Let $\lambda$ be as in Corollary \prettyref{cor:prob-density}. Let $f: \R^{n-1} \times \Gamma \backslash G \to \R$ be compactly supported and continuous. Let $f_t: \R^{n-1} \times \Gamma \backslash G \to \R$ be a family of continuous functions all supported on a compact set such that $f_t \to f$ uniformly. Then for any $g \in G$
  \begin{multline}
    \lim_{t\to \infty} e^{(n-1-\delta_{\Gamma})t} \int_{\R^{n-1}\times H \cap M} f_t(\vect{x}, \Gamma g \hor(\vect{x})m a_t) d\mu^{Haar}_{H\cap M}(m)d\lambda(\vect{x})\\ = \frac{1}{|\mathrm{m}^{\operatorname{BMS}}|} \int_{\R^{n-1} \times\Gamma \backslash G }\lambda'(\vect{x}) f(\vect{x}, \alpha) d\mathrm{m}^{BR}(\alpha)d\omega_{\Gamma g \overline{H}}^{PS}(\vect{x})
  \end{multline}
\end{theorem}

\begin{proof}\phantom{\qedhere}

  Let $\mathcal{S} \subset \Gamma \backslash G:= \{ \alpha \in \Gamma \backslash G : \exists t>0, \vect{x} \in \R^{n-1} \mbox{ s.t } f_t(\vect{x},\alpha) \neq 0\}$ (which is compact as the support of the entire family $f_t$ is compact) and let $\zeta(\alpha)$ be a smooth compactly supported bump function equal to $1$ on $\mathcal{S}$. As $f_t$ converges to $f$ uniformly and all functions are uniformly continuous, for all $\delta>0$ there exist $\epsilon = \epsilon(\delta)>0$ and $t_0>0$ such that for all $\vect{x}_0 \in \R^{n-1}$

  \begin{align}\label{f delta estimate}
    \begin{aligned}
    f(\vect{x}_0,g)-\delta \zeta(g) &\le f(\vect{x},g) \le f(\vect{x}_0,g)+\delta \zeta(g) \\
    f(\vect{x}_0,g)-\delta \zeta(g) &\le  f_t(\vect{x},g)  \le f(\vect{x}_0,g)+\delta \zeta(g)
    \end{aligned}
  \end{align}
  for all $\vect{x} \in \vect{x}_0 + [0,\epsilon)^{n-1}$ and $t>t_0$. We fix $\delta>0$ and let $\epsilon = \epsilon(\delta)$ to be adjusted later in the proof) and decompose $\R^{n-1}$ as follows

  \begin{align}
    &\int_{H\cap M}\int_{\R^{n-1}} f_t(\vect{x},\Gamma g \hor(\vect{x})ma_t) d\lambda(\vect{x})d\mu^{Haar}_{H\cap M}(m) \notag \\
   &\phantom{=} = \sum_{\vect{k} \in \Z^{n-1}}\int_{H\cap M}\int_{\epsilon \vect{k}+[0,\epsilon)^{n-1}} f_t(\vect{x},\Gamma g \hor(\vect{x})m a_t) d\lambda(\vect{x})d\mu^{Haar}_{H\cap M}(m) \label{old 3.5}\\
    &\phantom{=}\le  \sum_{\vect{k} \in \Z^{n-1}} \int_{H\cap M}\int_{\epsilon \vect{k}+[0,\epsilon)^{n-1}} f(\epsilon\vect{k},\Gamma g \hor(\vect{x})m a_t) + \delta \zeta(\Gamma g \hor(\vect{x})m a_t) d\lambda(\vect{x})d\mu^{Haar}_{H\cap M}(m) \notag
  \end{align}
  For each $\vect{k}$ and $\mathcal{E}_{\vect{k}} = \epsilon \vect{k} + [0,\epsilon)^{n-1}$ we can apply Corollary \ref{cor:prob-density} to the r.h.s of \eqref{old 3.5}, and then use that the $\zeta$ has compact support to conclude:
    \begin{align}
      \begin{aligned}
      &\lim_{t\to \infty} e^{(n-1-\delta_{\Gamma})t} \int_{H\cap M}\int_{\mathcal{E}_{\vect{k}}} f_t(\epsilon\vect{k},\Gamma g\hor(\vect{x}) ma_t) d\lambda(\vect{x})d\mu^{Haar}_{H\cap M}(m) \\ 
      &\phantom{====}\le \frac{1}{|\mathrm{m}^{BMS}|} \int_{\mathcal{E}_{\vect{k}} \times \Gamma\backslash G }\lambda'(\vect{x}) (f(\vect{x},\alpha)+2\delta \zeta(\alpha)) d\mathrm{m}^{BR}(\alpha)d\omega^{PS}_{\Gamma, g,\overline{H}}(\vect{x}) \\
      &\phantom{====}= \frac{1}{|\mathrm{m}^{BMS}|} \int_{ \mathcal{E}_{\vect{k}}\times \Gamma\backslash G}\lambda'(\vect{x}) f(\vect{x},\alpha) d\mathrm{m}^{BR}(\alpha)d\omega^{PS}_{\Gamma, g,\overline{H}}(\vect{x}) + C_{\vect{k}} \delta.\label{pomp 4}
      \end{aligned}
  \end{align}
   Since $\delta\int_{\R^{n-1} \times \Gamma \backslash G }\lambda^{\prime}(\vect{x}) \zeta(\alpha) d\mathrm{m}^{BR}(\alpha) d\omega^{PS}_{\Gamma,g,\overline{H}}(\vect{x})< \infty$ we know that $\sum_{\vect{k} \in \Z^{n-1}}C_{\vect{k}}<\infty$.  Putting this all together we get, that there exists a $C<\infty$ such that for any $\delta>0$,

  \begin{align}\label{eqn:intermediate limsup}
    \begin{aligned}
    \limsup_{t\to \infty} \; e^{(n-1-\delta_{\Gamma})t} \int_{ \R^{n-1} \times H \cap M } &f_t(\vect{x}, \Gamma g \hor(\vect{x})m a_t) d\mu^{Haar}_{H\cap M}(m) d\lambda(\vect{x}) \\
     \le \frac{1}{|\mathrm{m}^{\operatorname{BMS}}|}& \int_{\R^{n-1} \times \Gamma \backslash G} \lambda'(\vect{x})f(\vect{x}, \alpha)d\mathrm{m}^{BR}(\alpha) d\omega^{PS}_{\Gamma, g,\overline{H}}(\vect{x}) +C\delta. 
     \end{aligned}
  \end{align}


  A similar lower bound can be achieved for the $\liminf$ from which the Theorem follows.

\end{proof}

For a given $t_0>0$, let $\{ \mathcal{E}_t\}_{t\ge t_0}$ be bounded subsets of $ \R^{n-1} \times \Gamma \backslash G  $ all with boundary of $\omega^{PS}_{\Gamma, g, \overline{H}}\times \mathrm{m}^{BR}$-measure $0$, and define

\begin{eqnarray}
  \lim \left( \inf \mathcal{E}_t \right)^o &:=& \bigcup_{t \ge t_0} \left( \bigcap_{s \ge t} \mathcal{E}_s \right)^o \\
  \lim \overline{ \sup \mathcal{E}_t} &:=& \bigcap_{t \ge t_0} \overline{ \bigcup_{s \ge t} \mathcal{E}_s }\\
  \lim  \sup \mathcal{E}_t &:=& \bigcap_{t \ge t_0}  \bigcup_{s \ge t} \mathcal{E}_s 
\end{eqnarray}

In which case it is possible to prove a similar corollary to {\cite[Theorem 5.6]{MarklofStrom2010}} (with the exception that, as the $\mathrm{m}^{BR}$ is not finite on $\Gamma\backslash G$ we require our sets to be uniformly bounded):


\begin{corollary}\label{cor:test-functions}
  Let $\lambda$ be a Borel probability measure on $\R^{n-1}$ as in Corollary \prettyref{cor:prob-density}. Then for any bounded family of subsets $\mathcal{E}_t \subset \R^{n-1} \times \Gamma \backslash G$ all with boundary of $ \omega^{PS}_{\Gamma, g, \overline{H}} \times\mathrm{m}^{BR}$-measure $0$, for any $g \in \Gamma \backslash G$

  \begin{multline}
     \liminf_{t \to \infty} e^{(n-1-\delta_{\Gamma})t} \int_{\R^{n-1}}\int_{M\cap H}\chi_{\mathcal{E}_t} (\vect{x}, \Gamma g \hor(\vect{x})m a_t) d\lambda(\vect{x})d\mu^{Haar}_{M\cap H}(m)\\ 
     \ge \frac{1}{|\mathrm{m}^{BMS}|} \int_{\R^{n-1} \times \Gamma \backslash G }\lambda'(\vect{x})\chi_{\lim(\inf \mathcal{E}_t)^o}(\vect{x},\alpha) d\mathrm{m}^{BR}(\alpha) d\omega^{PS}_{\Gamma, g, \overline{H}}(\vect{x})
  \end{multline}
  \begin{multline}
     \limsup_{t \to \infty} e^{(n-1-\delta_{\Gamma})t} \int_{\R^{n-1}}\int_{M\cap H}\chi_{\mathcal{E}_t} (\vect{x}, \Gamma g\hor(\vect{x})m a_t) d\lambda(\vect{x})d\mu^{Haar}_{M\cap H}(m)\\ 
     \le \frac{1}{|\mathrm{m}^{BMS}|} \int_{ \R^{n-1}\times \Gamma \backslash G }\lambda'(\vect{x})\chi_{\lim\overline{\sup \mathcal{E}_t}}(\vect{x},\alpha) d\mathrm{m}^{BR}(\alpha) d\omega^{PS}_{\Gamma, g, \overline{H}}(\vect{x}) \label{eqn:limsup}
  \end{multline}
  Moreover, if  $\lim\overline{\sup\mathcal{E}_t} \setminus \lim(\inf\mathcal{E}_t)^o$ has $\omega^{PS}_{\Gamma, g, \overline{H}} \times \mathrm{m}^{BR}$-measure $0$ then
  \begin{multline}
     \lim_{t \to \infty} e^{(n-1-\delta_{\Gamma})t} \int_{\R^{n-1}}\int_{M\cap H}\chi_{\mathcal{E}_t} (\vect{x}, \Gamma g \hor(\vect{x})m a_t) d\lambda(\vect{x})d\mu^{Haar}_{M\cap H}(m)\\ = \frac{1}{|\mathrm{m}^{BMS}|} \int_{ \R^{n-1} \times \Gamma \backslash G }\lambda'(\vect{x})\chi_{\lim\sup \mathcal{E}_t}(\vect{x},\alpha) d\mathrm{m}^{BR}(\alpha) d\omega^{PS}_{\Gamma, g, \overline{H}}(\vect{x})
\end{multline}

\end{corollary}

\begin{proof}\phantom{\qedhere}

 This Corollary follows from \prettyref{thm:multiple-functions} in exactly the same way as {\cite[ Theorem 5.6]{MarklofStrom2010}}, with one exception. Addressing only (\ref{eqn:limsup}) (as the other results follow similarly). Let

 \begin{equation}
   \tilde{\mathcal{E}}_t := \overline{\bigcup_{s \ge t}\mathcal{E}_s},
 \end{equation}
 thus $\mathcal{E}_t \subset \tilde{\mathcal{E}}_t \subset \tilde{\mathcal{E}}_{t_1}$ for $t \ge t_1$. Hence

 \begin{multline}
   \limsup_{t\to \infty} e^{(n-1-\delta_{\Gamma})t} \int_{\R^{n-1}}\int_{M \cap H} \chi_{\mathcal{E}_t}(\vect{x},\Gamma g \hor(\vect{x})m a_t) d\lambda(\vect{x}) d\mu^{Haar}_{H\cap M}(m)\\
   \le \limsup_{t_1\to \infty} \limsup_{t\to \infty} e^{(n-1-\delta_{\Gamma})t}\int_{\R^{n-1}}\int_{M\cap H}\chi_{\tilde{\mathcal{E}}_{t_1}}(\vect{x},\Gamma g \hor(\vect{x})ma_t) d\lambda(\vect{x})d\mu^{Haar}_{H\cap M}(m).\label{eqn:limit t_1}
 \end{multline}
 From here we apply \prettyref{thm:multiple-functions} for a fixed $f=f_t = \chi_{\mathcal{E}_{t_1}}$ by approximating compactly supported characteristic functions with bounded, compactly supported, continuous functions. That is, consider

 \begin{multline}
   \left| \limsup_{t\to \infty} e^{(n-1-\delta_{\Gamma})t}\int_{H\cap M} \int_{\R^{n-1}} \chi_{\tilde{\mathcal{E}}_{t_1}}(\vect{x}, \Gamma g \hor(\vect{x})m a_t) d\lambda d\mu^{Haar}_{M\cap H} \right.
   \\ \left.-\frac{1}{|\mathrm{m}^{BMS}|} \int_{\R^{n-1} \times \Gamma \backslash G }\lambda'(\vect{x})  \chi_{\tilde{\mathcal{E}}_{t_1}}(\vect{x},\alpha) d\mathrm{m}^{BR}(\alpha) d\omega_{\Gamma, g, \overline{H}}^{PS}(\vect{x})\right|. \label{eqn: difference}
 \end{multline}
 Fix $\epsilon >0$ and let $\phi$ be a bounded, compactly supported function such that $\phi = \chi_{\tilde{\mathcal{E}}_{t_1}}$ outside of a $\delta$-neighborhood of the boundary of $\tilde{\mathcal{E}}_{t_1}$. $\delta = \delta(\epsilon)>0$ will be fixed later in the proof. Write

 \begin{align}
   \begin{aligned}
   &\mbox{\eqref{eqn: difference}} = \left| \limsup_{t \to \infty} \left(e^{(n-1-\delta_{\Gamma})t}\right.\right.\left. \int_{H \cap M } \int_{\R^{n-1}} \chi_{\tilde{\mathcal{E}}_{t_1}}(\vect{x},\Gamma g \hor(\vect{x})m a_t) \right. \\
                                            &\phantom{=======}\left.+ \phi(\vect{x},\Gamma g \hor(\vect{x})m a_t) - \phi(\vect{x},\Gamma g \hor(\vect{x})m a_t)\right) d\lambda d\mu^{Haar}_{M\cap H}\\
                                    &\phantom{===========}- \left. \frac{1}{\BMS} \int_{\R^{n-1}\times \Gamma \backslash G}\lambda^{\prime}(\vect{x}) \chi_{\tilde{\mathcal{E}}_{t_1}}(\vect{x},\alpha) d\mathrm{m}^{BR}(\alpha) d\omega_{\Gamma, g, \overline{H}}^{PS}(\vect{x})\right|.
    \end{aligned}
   \end{align}
   Applying Theorem \ref{thm:multiple-functions} to the second term in the second line then gives that \eqref{eqn: difference} is less than or equal

   \begin{align}\label{eqn:limit E_t}
     \begin{aligned}
       &\mbox{\eqref{eqn: difference}} \le \limsup_{t\to \infty} e^{(n-1-\delta_{\Gamma})t}\int_{H\cap M} \int_{\R^{n-1}} \left| \chi_{\tilde{\mathcal{E}}_{t_1}}(\vect{x}, \Gamma g \hor(\vect{x})m a_t)\right.\\
       &\phantom{++++++++++++++++++++}-\left.\phi (\vect{x}, \Gamma g \hor(\vect{x})m a_t)\right| d\lambda d\mu^{Haar}_{M \cap H}(m)\\
      &\phantom{++++++}+\frac{1}{|\mathrm{m}^{BMS}|} \int_{ \R^{n-1}\times \Gamma \backslash G }\lambda'(\vect{x}) \left| \chi_{\tilde{\mathcal{E}}_{t_1}}(\vect{x},\alpha) -\phi(\vect{x},\alpha) \right| d\mathrm{m}^{BR}(\alpha) d\omega_{\Gamma, g, \overline{H}}^{PS}(\vect{x}).
    \end{aligned}
 \end{align}
 Now let $\tilde{\phi}$ be a continuous, bounded, function supported on the $\delta$-neighborhood of $\tilde{\mathcal{E}}_{t_1}$ such that $\tilde{\phi} \ge \left|\chi_{\tilde{\mathcal{E}}_{t_1}} - \phi\right|$ everywhere. Hence

 \begin{multline}
   \mbox{\eqref{eqn: difference}} \le \limsup_{t\to \infty} e^{(n-1-\delta_{\Gamma})t}\int_{H\cap M} \int_{\R^{n-1}}  \tilde{\phi} (\vect{x}, \Gamma g \hor(\vect{x})m a_t) d\lambda d\mu^{Haar}_{M \cap H}(m)
   \\  +\frac{1}{|\mathrm{m}^{BMS}|} \int_{ \R^{n-1}\times \Gamma \backslash G }\lambda'(\vect{x}) \left| \chi_{\tilde{\mathcal{E}}_{t_1}}(\vect{x},\alpha) -\phi(\vect{x},\alpha) \right| d\mathrm{m}^{BR}(\alpha) d\omega_{\Gamma, g, \overline{H}}^{PS}(\vect{x}).
 \end{multline}
 Now we may apply Theorem \ref{thm:multiple-functions} once again to $\tilde{\phi}$ to conclude

 \begin{align}\label{last step}
   \begin{aligned}
     &\mbox{\eqref{eqn: difference}} \le \\
     &\phantom{+++}\frac{1}{|\mathrm{m}^{BMS}|} \int_{ \R^{n-1}\times \Gamma \backslash G }\lambda'(\vect{x}) \left(\tilde{\phi}(\vect{x},\alpha) +  \left| \chi_{\tilde{\mathcal{E}}_{t_1}}(\vect{x},\alpha) -\phi(\vect{x},\alpha) \right|\right) d\mathrm{m}^{BR}(\alpha) d\omega_{\Gamma, g, \overline{H}}^{PS}(\vect{x}).
   \end{aligned}
 \end{align}
 Now note that by assumption the Patterson-Sullivan measure is finite and the Burger-Roblin measure is finite on bounded subsets. Since both terms in the integrand are bounded and supported on the $\delta$-neighborhood of $\tilde{\mathcal{E}}_{t_1}$, we may choose $\delta$ small enough such that the right hand side of \eqref{last step} is less than $\epsilon$. (\ref{eqn:limsup}) then follows from (\ref{eqn:limit t_1}) from which it follows that (\ref{eqn:limit E_t}) is less than $C\epsilon$ for some $C<\infty$.

The rest of the Theorem follows similarly.

\end{proof}

\section{Observer at Infinity}

Our goal is to consider observers inside hyperbolic half-space but it will be more convenient to first consider an observer on the boundary (w.l.o.g at $\vect{\infty}$) as this will allow us to use the horospherical equidistribution theorem stated above. Consider the projection of $\Gamma\vect{w}$ onto a horosphere centered at $\vect{\infty}$. Hence there are two situations, either $\vect{\infty}$ is the location of a cusp in a fundamental domain of $\Gamma$, or it is in a funnel. We will treat these two situations together.

 Consider the cusp with rank $0 \le l \le n-1$ at $\vect{\infty}$ (a rank $0$ cusp is trivial and hence describes the situation with no cusp). $\Gamma$ contains the (possibly trivial) subgroup $\Gamma_{\vect{\infty}}$. We may furthermore write

\begin{equation}
  \Gamma_{\vect{\infty}} = \{n_+(\vect{m}) : \vect{m} \in \mathcal{L} \},
\end{equation}
where $\mathcal{L}$ is a (possibly trivial) discrete subgroup of $\R^{n-1}$ of rank $l$. 

Define

\begin{equation}
  \mathcal{P}_{t,s}^{\vect{\infty}}(\overline{\vect{w}}) := \{ \operatorname{Re}(\gamma\vect{w}) \mbox{ mod }\mathcal{L} : \gamma \in \Gamma_{\vect{\infty}}  \backslash \Gamma / \Gamma_w, \; e^{-t} \le \operatorname{Im}(\gamma\vect{w}) < e^{s-t} \}. 
\end{equation}
and take $\mathcal{P}_{t,s}^{\vect{\infty}}(\overline{\vect{w}})$ to be a subset of a horospherical subgroup $\overline{H}$ by identifying $\overline{H}$ with $\R^{n-1}$ via group isomorphism $\hor$. 

The first order statistics for a boundary observer are given by:


\begin{theorem} \label{thm:OhShahAsymptotics}

  In the present context. Let $F \subset \overline{H}$ be a Borel subset  of the horosperical subgroup, $\overline{H}$, with $\mu_{\overline{H}}^{PS}(F)<\infty$ and $\mu_{\overline{H}}^{PS}(\partial F) = 0$. Then the following asymptotic formula holds as $t \to \infty$
  \begin{equation} \label{eqn:P-asymptotics}
    \# (\mathcal{P}_{t,\infty}^{\vect{\infty}}(\overline{\vect{w}})\cap F) \sim  \vartheta \mu_{\overline{H}}^{PS}(F) e^{\delta_{\Gamma} t}
  \end{equation}
  for $\vartheta$ defined below \eqref{cusp cone} depending only on $\Gamma$. 

\end{theorem}

\begin{remark} Asymptotic formulas for the number of lattice points in balls and sectors have been studied previous, for example by Good \cite{Good1983}. Bourgain-Kontorovich-Sarnak \cite{BourgainKontorovichSarnak2010} described the asymptotics of orbit points in growing balls when the critical exponent is less than $1/2$ in dimension $n=2$. Oh and Shah \cite{OhShah2013} then extended these results to full generality, including the sector case. This theorem concerns horospherical sectors which is also covered by Oh and Shah (see {\cite[Theorem 7.16]{OhShah2013}}).
\end{remark}

Consider the following rescaled test sets in $\T^{l} \times \R^{n-1-l}$ (scaled to match the scaling in (\ref{eqn: D volume}))

\begin{equation}
  \mathcal{B}_{t,s}(\mathcal{A},\vect{x}) = N_{t,s}(\overline{\vect{w}})^{-1/\delta_{\Gamma}}\mathcal{A} - \vect{x} + \mathcal{L} \subset \T^{l} \times \R^{n-1-l}, \label{eqn:definition-B}
\end{equation}
where $N_{t,s}(\overline{\vect{w}}) := \# \mathcal{P}_{t,s}^{\vect{\infty}}(\overline{\vect{w}})$ and $\mathcal{A} \subset \R^{n-1}$ is bounded. The base point $\vect{x}$ will be chosen with law $\lambda$. Let

\begin{equation}
  \mathcal{N}_{t,s}^{\vect{\infty}}(\mathcal{A},\vect{x}; \overline{\vect{w}}) := \# (\mathcal{P}_{t,s}^{\vect{\infty}}(\overline{\vect{w}}) \cap \mathcal{B}_{t,s}(\mathcal{A},\vect{x})).
\end{equation}

Let $\mathcal{A}_1,...,\mathcal{A}_m$ be bounded test sets with boundary of Lebesgue measure $0$. Given a compactly supported density $\lambda'$ on $\T^{l} \times \R^{n-1-l}$ write

\begin{equation}
  A_{\lambda} = \int_{\T^{l} \times \R^{n-1-l}} \lambda(\vect{x}) d\omega_{\Gamma, \overline{H}}^{PS}(\vect{x})
\end{equation}
($\omega_{\Gamma,\overline{H}}^{PS} := \omega_{\Gamma,g,\overline{H}}^{PS}$ with $g = Id$ the identity).


\begin{theorem} \label{thm:joint-limit}

  Let $\lambda$ be a compactly supported Borel probability measure on $\T^{l} \times \R^{n-1-l}$ absolutely continuous with respect to Lebesgue measure, with continuous density. Then for any $r =(r_1,...,r_m) \in \Z_{>0}^m$, $s \in (0,\infty]$ and $\mathcal{A} = \mathcal{A}_1\times... \times \mathcal{A}_m$

  \begin{equation}
    E_{s}(r,\mathcal{A}; \overline{\vect{w}}) := \lim_{t\to \infty}e^{(n-1-\delta_{\Gamma})t} \lambda (\{\vect{x} \in \T^{l} \times \R^{n-1-l}: \mathcal{N}^{\vect{\infty}}_{t,s}(\mathcal{A}_j, \vect{x}; \overline{\vect{w}}) = r_j, \forall j \})
  \end{equation}
  exists and is given by

  \begin{equation}
     E_{s}(r,\mathcal{A};\overline{\vect{w}}) = \frac{A_{\lambda}}{|\mathrm{m}^{BMS}|}\mathrm{m}^{BR}(\{ \alpha \in \Gamma \backslash G : \# (\alpha^{-1}\overline{\vect{w}} \cap \mathcal{Z}(s,\mathcal{A}_j))=r_j \forall j \}),
  \end{equation}
  with 
  
  \begin{equation}
    \mathcal{Z}(s,\mathcal{A}_j):= \{ \vect{z} \in \half^n : \operatorname{Re}\vect{z} \in \vartheta^{-1/\delta_{\Gamma}}\mathcal{A}_j, 1 \le \operatorname{Im}\vect{z} < e^s \}. \label{eqn:Z-definition} 
  \end{equation}
  Moreover, $E_{s}(r,\mathcal{A};\overline{\vect{w}})$ is continuous in $s$ and $\mathcal{A}$.

\end{theorem}
Borrowing notation from \cite{MarklofVinogradov2018}, by continuous in the set $\mathcal{A}$ we mean that there exists a constant $C<\infty$ such that

\begin{equation}  \label{eqn:continuous-in-A}
  |E_{s}(r,\mathcal{A};\overline{\vect{w}}) - E_{s}(r,\mathcal{B};\overline{\vect{w}})| \le C \operatorname{vol}_{\R^{m(n-1)}}(\mathcal{B} \setminus \mathcal{A})
\end{equation}
for any two sets $\mathcal{A} \subset \mathcal{B} \subset \R^{m(n-1)}$ as in \prettyref{thm:joint-limit}.

With the exception of the proof of Proposition \ref{prop:MV-Lemma-5} and some other details, the proof of \prettyref{thm:joint-limit} follows the same lines as proof of {\cite[Theorem 4]{MarklofVinogradov2018}}. However we will include many details covered there for completeness

For a set $\mathcal{A} \subset \half^n$ with boundary of $BR$-measure $0$ (i.e $\BR(\pi^{-1}(\partial\mathcal{A})) = 0$) and $r \in \Z_{> 0}$ define the following sets

\begin{eqnarray}
  [\mathcal{A}]_{\le r} &:=& \{ \alpha \in \Gamma \backslash G : 0< \#(\mathcal{A} \cap \alpha^{-1}\overline{\vect{w}}) \le r \}\\
  \left[\mathcal{A}\right]_{\ge r} &:=& \{ \alpha \in \Gamma \backslash G : \#(\mathcal{A} \cap \alpha^{-1}\overline{\vect{w}}) \ge r \}\\
  \left[\mathcal{A}\right]_{=r} &:=& \{ \alpha \in \Gamma \backslash G : \#(\mathcal{A} \cap \alpha^{-1}\overline{\vect{w}}) = r \}
\end{eqnarray}


\begin{proposition} \label{prop:MV-Lemma-5}
  Consider a measurable set with finite volume and boundary of $BR$-measure $0$, $\mathcal{B} \subset \half^n$ such that $\inf \{t :  n_+a_{-t}\vect{i} \in \mathcal{B}\} = t_0 >-\infty$ and $\mathcal{A} \subset \mathcal{B}$ (also with boundary of $BR$-measure $0$). In that case, with $\vect{w}=g_{\vect{w}}\vect{i}$ and $r \in \N_{>0}$

  \begin{equation}
    \mathrm{m}^{BR}(\left[\mathcal{A}\right]_{\ge 1}) \le    \frac{C_{t_0}}{\#\Gamma_{\vect{w}}} \operatorname{vol}_{\half^n} ( g_{\vect{w}}^{-1}\mathcal{A}) ,
  \end{equation}

  \begin{equation}
    |\mathrm{m}^{BR}(\left[\mathcal{A}\right]_{= r}) - \mathrm{m}^{BR}(\left[\mathcal{B}\right]_{= r})|  \le    \frac{C_{t_0}}{\#\Gamma_{\vect{w}}} \operatorname{vol}_{\half^n} ( g_{\vect{w}}^{-1}(\mathcal{B} \setminus \mathcal{A})) ,
  \end{equation}
  and
  \begin{equation}
    0 \le \mathrm{m}^{BR}(\left[\mathcal{A}\right]_{\le r}) - \mathrm{m}^{BR}(\left[\mathcal{B}\right]_{\le r})  \le    \frac{C_{t_0}}{\#\Gamma_{\vect{w}}} \operatorname{vol}_{\half^n} (g_{\vect{w}}^{-1}( \mathcal{B} \setminus \mathcal{A})),
  \end{equation}
  with $C_{t_0}< \infty$ depending on $t_0$ and $\vect{w}$.
\end{proposition}


\begin{proof}\phantom{\qedhere}

  The proof of this Lemma will follow from a Siegel type estimate. Consider
  
  \begin{equation}
    \int_{G} \chi_{\mathcal{A}}(\alpha^{-1} \vect{w}) d\mathrm{m}^{BR}(\alpha)
  \end{equation}
  By making the change of variables $ \alpha \mapsto g_{\vect{w}} \alpha g_{\vect{w}}^{-1}$ we can then consider the Burger-Roblin measure associated to the group $\Gamma^{\vect{w}} := g_{\vect{w}} \Gamma g_{\vect{w}}^{-1}$. Thus

  \begin{equation}
    \int_{G} \chi_{\mathcal{A}}(\alpha^{-1} \vect{w}) d\mathrm{m}^{BR}(\alpha) = \int_{G} \chi_{\mathcal{A}}(g_{\vect{w}} \alpha^{-1} \vect{i}) dm_{\Gamma^{\vect{w}}}^{BR}(\alpha).
  \end{equation}
  The decomposition of the Burger-Roblin measure from {\cite[Proposition 7.3]{OhShah2013}} together with the fact $\chi_{\mathcal{A}} \in \mathcal{C}(T^{1}(\half^n)$ give

  \begin{equation}
    \int_{G} \chi_{\mathcal{A}}(\alpha^{-1} \vect{w}) d\mathrm{m}^{BR}(\alpha)= \int_{KAN_+} \chi_{g_{\vect{w}}^{-1}\mathcal{A}}( (ka_tn_+)^{-1}  \vect{i})e^{-\delta_{\Gamma}t} d\mu^{Haar}_{N_+}(n_+) dt d\nu^{\vect{w}}_{\vect{i}}(k  \vect{X}_{\vect{i}}^-),
  \end{equation}
  $\nu^{\vect{w}}_{\vect{i}}(k \vect{X}_{\vect{i}}^-)$ is the conformal density of dimension $\delta_{\Gamma} = \delta_{\Gamma^{\vect{w}}}$ supported on $\Lambda(\Gamma^{\vect{w}})$. Applying the inverse inside the bracket and recalling that $K$ is the stabilizer of $\vect{i}$ gives

  \begin{equation}
    \int_{G} \chi_{\mathcal{A}}(\alpha^{-1} \vect{w}) d\mathrm{m}^{BR}(\alpha)= \left|\nu^{\vect{w}}_{\vect{i}}\right| \int_{AN_+} \chi_{g_{\vect{w}}^{-1}\mathcal{A}}( n_+^{-1}a_{-t}  \vect{i})e^{-\delta_{\Gamma}t} d\mu^{Haar}_{N_+}(n_+) dt.
  \end{equation}
  As the integral on $N_+$ is with respect to Haar measure we can change variables giving
  
  \begin{equation}
   \int_{G} \chi_{\mathcal{A}}(\alpha^{-1} \vect{w}) d\mathrm{m}^{BR}(\alpha) \le \tilde{C}_{t_0} \int_{AN_+} \chi_{g_{\vect{w}}^{-1}\mathcal{A}}( n_+a_{-t}  \vect{i}) d\mu^{Haar}_{N_+}(n_+) dt.
  \end{equation}
  with $\tilde{C}_{t_0}= \left|\nu^{\vect{w}}_{\vect{i}}\right|e^{-\delta_{\Gamma}t_0}$. Now if we fix the notation $n_+(\vect{x}) : = \begin{psmallmatrix} 1 & \vect{x}\\0 & 1 \end{psmallmatrix}$ then the measure $d\mu^{Haar}_{N_+}(n_+(\vect{x})) = d\vect{x}$. Thus

  \begin{align}
    \begin{aligned}\label{eqn:thin-Siegel}
    \int_{G} \chi_{\mathcal{A}}(\alpha^{-1}\vect{w}) d\mathrm{m}^{BR}(\alpha)&\le  \tilde{C}_{t_0} \int_{\R^{n-1} \times \R} \chi_{g_{\vect{w}}^{-1}\mathcal{A}}(a_{-t} n_+(e^{-t} \vect{x})  \vect{i}) d\vect{x} dt\\
                                                                             & \le \tilde{C}_{t_0} \int_{\R^{n-1} \times \R} e^{-(n-1)t}\chi_{g_{\vect{w}}^{-1}\mathcal{A}}(a_{t} n_+( \vect{x})  \vect{i}) d\vect{x} dt\\
                                                                             & \le C_{t_0} \vol_{\half^n}(g_{\vect{w}}^{-1}\mathcal{A}),
    \end{aligned}
  \end{align}
  with $C_{t_0} = \left|\nu^{\vect{w}}_{\vect{i}}\right|e^{(-n+1 - \delta_{\Gamma})t_0}$.

  The proof of the proposition now follows from (\ref{eqn:thin-Siegel}), Chebyshev's inequality and some simple set manipulations (see {\cite[Lemma 5]{MarklofVinogradov2018})} and is simply a consequence of the following

\begin{eqnarray}
  \int_{\Gamma \backslash G} \# (\mathcal{A} \cap \alpha^{-1}\overline{\vect{w}}) d\mathrm{m}^{BR}(\alpha) &=& \int_{\Gamma \backslash G} \sum_{\gamma \in \Gamma / \Gamma_{\vect{w}}} \chi_{\mathcal{A}} ( \alpha^{-1} \gamma \vect{w} ) d\mathrm{m}^{BR}(\alpha) \\
                                                                &=&  \frac{1}{\#\Gamma_{\vect{w}}}\int_{G} \chi_{\mathcal{A}} (  \alpha^{-1}\vect{w}) d\mathrm{m}^{BR}(\alpha) .
\end{eqnarray}

\end{proof}


\begin{lemma} \label{lem:MV-Lemma-6}
  Under the hypothesis of \prettyref{thm:joint-limit}, given an $\epsilon >0$ there exists a $t_0\in \R$ and bounded sets $\mathcal{A}^-_j, \mathcal{A}^+_j \subset \R^{n-1} $ with boundary of Lebesgue measure $0$ such that
  \begin{equation}
    \mathcal{A}^-_j \subset \mathcal{A}_j \subset \mathcal{A}^+_j,
  \end{equation}

  \begin{equation}
    \operatorname{vol}_{\R^{n-1}}(\mathcal{A}^+_j \setminus \mathcal{A}^-_j) < \epsilon
  \end{equation}
  and for all $t \ge t_0$

  \begin{equation}
    \#(a_tn_+ (\vect{x})\overline{\vect{w}}  \cap \mathcal{Z}(s,\mathcal{A}^-_j)) \le \mathcal{N}_{t,s}^{\vect{\infty}}(\mathcal{A}_j, \vect{x}; \overline{\vect{w}})\\ \le \#(a_tn_+ (\vect{x})\overline{\vect{w}} \cap \mathcal{Z}(s,\mathcal{A}^+_j))
  \end{equation}

\end{lemma}


\begin{proof}\phantom{\qedhere}
  Similarly to {\cite[Lemma 6]{MarklofVinogradov2018}} we write

  \begin{equation}
    \mathcal{N}_{t,s}^{\vect{\infty}}(\mathcal{A}_j, \vect{x}; \overline{\vect{w}}) = \# (a_tn_+ (\vect{x})\overline{\vect{w}} \cap \mathcal{Z}(s, e^t\vartheta^{1/\delta_{\Gamma}}N_{t,s}(\overline{\vect{w}})^{-1/\delta_{\Gamma}}\mathcal{A}_j))
  \end{equation}
  and note that $e^t \vartheta^{1/\delta_{\Gamma}}N_{t,s}(\overline{\vect{w}})^{-1/\delta_{\Gamma}} \to 1$ from which the lemma follows. 

\end{proof}

Furthermore the analogue of {\cite[Lemma 7]{MarklofVinogradov2018}} applies in this context as well. Suppose $-\infty < a< b \le \infty$ and $\mathcal{A} \subset \R^{n-1}$, define
 
\begin{equation}\label{Z(3) def}
  \mathcal{Z}(a,b,\mathcal{A}):= \{ \vect{z} \in \half^n : \operatorname{Re}\vect{z} \in \vartheta^{-1/\delta_{\Gamma}}\mathcal{A}, e^a \le \operatorname{Im} \vect{z} \le e^b \}
\end{equation}


\begin{lemma} \label{lem:MV-Lemma-7}
  Under the hypothesis of \prettyref{thm:joint-limit}, for all $s \ge 0$ we have
  \small
  \begin{multline}\label{eqn:MV-Lemma-7}
    \limsup_{t\to\infty} e^{(n-1-\delta_{\Gamma})t} \left| \lambda (\{ \vect{x} \in \T^{l}\times \R^{n-1-l} : 0< \#(a_tn_+(\vect{x})\overline{\vect{w}}\cap \mathcal{Z}(\infty,\mathcal{A}_j)) \le r_j, \; \forall j \})\right.\\
    -\left. \lambda (\{\vect{x} \in \T^{l}\times \R^{n-1-l} : 0<\#(a_tn_+(\vect{x})\overline{\vect{w}} \cap \mathcal{Z}(s,\mathcal{A}_j)) \le r_j \forall j \}) \right| \le C e^{-\delta_{\Gamma} s/2}(\vol_{\R^{n-1}} \tilde{\mathcal{A}})^{1/2},
  \end{multline}\normalsize
  where $\tilde{\mathcal{A}} = \bigcup_j \mathcal{A}_j$ and $C>0$ is some constant.
\end{lemma}


\begin{proof}\phantom{\qedhere}
  The left hand side of (\ref{eqn:MV-Lemma-7}) without the $\limsup$ is less than or equal to
 
  \begin{equation}\label{eqn:mv-lem}
    e^{(n-1-\delta_{\Gamma})t}\lambda (\{ \vect{x} \in \T^{l}\times \R^{n-1-l} : \#(a_tn_+(\vect{x})\overline{\vect{w}} \cap \mathcal{Z}(s,\infty,\tilde{\mathcal{A}})) \ge 1 \})
   \end{equation}
  and $\#(a_tn_+(\vect{x})\overline{\vect{w}} \cap \mathcal{Z}(s,\infty,\tilde{\mathcal{A}}) = \mathcal{N}_{t-s,\infty}^{\vect{\infty}}(\eta_{t-s}e^{t-s}\tilde{\mathcal{A}},\vect{x};\overline{\vect{w}})$, where $\eta_{t-s} = \frac{N_{t-s,\infty}^{1/\delta_{\Gamma}}}{\vartheta^{1/\delta_{\Gamma}}e^{t-s}} \to 1$ as $t \to \infty$

  Chebyshev's inequality then implies

  \begin{equation}
    \mbox{(\ref{eqn:mv-lem})} \le e^{(n-1-\delta_{\Gamma})t}\int_{\T^{l}\times \R^{n-1-l}}\mathcal{N}_{t-s,\infty}^{\vect{\infty}}(\eta_{t-s}e^{t-s}\tilde{\mathcal{A}},\vect{x};\overline{\vect{w}})d\lambda(\vect{x})
  \end{equation}
  Note further, by \prettyref{thm:OhShahAsymptotics}

  \begin{equation}
    \lim_{t\to \infty} e^{(n-1-\delta_{\Gamma})t}\int_{\T^{l}\times \R^{n-1-l}} \mathcal{N}_{t,s}^{\vect{\infty}} (\mathcal{A},\vect{x} ; \overline{\vect{w}}) d\vol (\vect{x}) = \vol_{\R^{n-1}} \mathcal{A}.
  \end{equation}
  Hence, for any $R\ge c$ for some constant $c>0$

  \begin{align}
    \begin{aligned}
    & e^{(n-1-\delta_{\Gamma})t}\int_{\T^{l}\times \R^{n-1-l}}\mathcal{N}_{t-s,\infty}^{\vect{\infty}}(\eta_{t-s}e^{t-s}\tilde{\mathcal{A}},\vect{x};\overline{\vect{w}})d\lambda(\vect{x})\\
    & \phantom{=======}\le R e^{(n-1-\delta_{\Gamma})t}\int_{\T^{l}\times \R^{n-1-l}}\mathcal{N}_{t-s,\infty}^{\vect{\infty}}(\eta_{t-s}e^{t-s}\tilde{\mathcal{A}},\vect{x};\overline{\vect{w}}) \chi_{\mbox{supp}(\lambda)}(\vect{x})d\vol(\vect{x})\\
    &\phantom{==================================}\to  R e^{-(n-1) s}\vol_{\R^{n-1}}\tilde{\mathcal{A}},
     \end{aligned}
  \end{align}
  as $t \to \infty$. Choosing

  \begin{equation}
    R :=  C\frac{e^{(n-1) s} (\vol_{\R^{n-1}}\tilde{\mathcal{A}})^{-1/2}}{e^{\delta_{\Gamma} s/2}} 
 \end{equation}
 proves the theorem (the constant $C$ is there to ensure $R>c$).

\end{proof}

\begin{proof}\phantom{\qedhere}[Proof of \prettyref{thm:joint-limit}]
  This proof is similar to {\cite[Proof of Theorem 4]{MarklofVinogradov2018}}. It suffices to show that for all $r=(r_1,...,r_m) \in \Z^m_{> 0}$ and all sets $\mathcal{A} = \mathcal{A}_1 \times ... \times \mathcal{A}_m$ with $\mathcal{A}_j \subset \R^{n-1}$ bounded with boundary of Lebesgue measure $0$ the following limit holds as $t \to \infty$

  \begin{multline}
    e^{(n-1-\delta_{\Gamma})t}\lambda(\{\vect{x}\in \T^{l}\times \R^{n-1-l} : 0<  \mathcal{N}_{t,s}^{\vect{\infty}} (\mathcal{A}_j; \vect{x}; \overline{\vect{w}}) \le r_j \forall j \}) \\
    \to \frac{A_{\lambda}}{|\mathrm{m}^{BMS}|} \mathrm{m}^{BR} (\{ \alpha \in \Gamma  \backslash G : 0< \#(\alpha^{-1} \overline{\vect{w}} \cap \mathcal{Z}(s,\mathcal{A}_j) \le r_j \forall j \}).
  \end{multline}
  The left hand side is equal

  \begin{equation}
    e^{(n-1-\delta_{\Gamma})t}\int_{\T^{l}\times \R^{n-1-l}}\chi_{\mathcal{E}_{t,s}}(\vect{x}, (a_tn_+(\vect{x}) )^{-1}) d\lambda(\vect{x})
  \end{equation}
  with

  \begin{equation}
    \mathcal{E}_{t,s} := \operatorname{supp}(\lambda) \times \{ \alpha \in \Gamma \backslash G : 0< \# (\alpha^{-1}\overline{\vect{w}}  \cap \mathcal{Z} (s, e^t \vartheta^{1/\delta_{\Gamma}} N_{t,s}(\overline{\vect{w}})^{-1/\delta_{\Gamma}}\mathcal{A}_j)) \le r_j \; \forall j \}.
  \end{equation}

  \underline{Assume $s < \infty$}: Fix $\epsilon>0 $\hspace{2mm} and approximate $\mathcal{A}$ by $\mathcal{A}^{\pm}$ as in \prettyref{lem:MV-Lemma-6} such that $\vol_{\R^{n-1}}(\mathcal{A}^+\setminus \mathcal{A}^-) < \epsilon$. Giving sets $\mathcal{E}^{\pm}_s$ such that $\mathcal{E}_s^+ \subset \mathcal{E}_{t,s} \subset \mathcal{E}_s^-$ for all $t \ge t_0$. In this case $\overline{\mathcal{E}_{t,s}}$ is compact. This follows by the compactness of $\lambda$ and boundedness of $\mathcal{Z}(s,e^t\vartheta^{1/\delta_{\Gamma}}N_{t,s}(\overline{\vect{w}})^{-1/\delta_{\Gamma}}\mathcal{A}_j)$. Hence we can apply Corollary \prettyref{cor:test-functions} by taking $\lambda$ to be a compactly supported probability measure on $\T^{l}\times \R^{n-1-l}$. 

  For this we require our test sets, $\mathcal{E}_{s}^{\pm}$, to have boundary of $ \omega^{PS}_{\Gamma, g, \overline{H}} \times \mathrm{m}^{BR} $-measure $0$. Note that $\lambda$ is absolutely continuous with respect to Lebesgue measure and has continuous density and  the Patterson-Sullivan density is non-atomic. Therefore $\mathcal{E}_{t,s}$ has boundary of $ \omega^{PS}_{\Gamma, g, \overline{H}} \times \mathrm{m}^{BR} $-measure $0$. Hence we can choose $\mathcal{A}^{\pm}$ such that the same is true of $\mathcal{E}_s^{\pm}$.

Hence we can apply Corollary \prettyref{cor:test-functions}. Giving
  
\begin{align}
  \begin{aligned}
    &\limsup_{t\to \infty}e^{(n-1-\delta_{\Gamma})t}\int_{\T^{l}\times \R^{n-1-l}} \chi_{\mathcal{E}_{t,s}}(\vect{x},n_+(-\vect{x})a_{-t})d\lambda(\vect{x}) \le \frac{A_{\lambda}}{|\mathrm{m}^{BMS}|} \mathrm{m}^{BR}(\overline{\mathcal{E}_s^-}),\\
    &\liminf_{t\to \infty}e^{(n-1-\delta_{\Gamma})t}\int_{\T^{l}\times \R^{n-1-l}} \chi_{\mathcal{E}_{t,s}}(\vect{x},n_+(-\vect{x})a_{-t})d\lambda(\vect{x}) \ge \frac{A_{\lambda}}{|\mathrm{m}^{BMS}|} \mathrm{m}^{BR}((\mathcal{E}_s^+)^o).
    \end{aligned}
  \end{align}
  Finally Proposition \prettyref{prop:MV-Lemma-5}, \prettyref{lem:MV-Lemma-6} and the fact $\mathcal{Z}(s,\mathcal{A}^{\pm}_j)$ is bounded for $s< \infty$ imply that
    
  \begin{equation}
    \lim_{\epsilon \to 0}\mathrm{m}^{BR}(\overline{\mathcal{E}_s^-} \setminus (\mathcal{E}_s^+)^o) = 0
  \end{equation}
  which proves \prettyref{thm:joint-limit} for $s<\infty$. 

  \underline{Assume $s=\infty$:} The equidistribution theorems stated in Section 3 hold only for compactly supported functions $\chi$. Hence an approximation arguement is needed to get around this. 

Consider

\begin{equation}
  \limsup_{t\to \infty}e^{(n-1-\delta_{\Gamma})t} \int_{\T^l \times \R^{n-1-l} } \chi_{\mathcal{E}_{t,\infty}}(\vect{x},n_+(-\vect{x})a_{-t}) d\lambda(\vect{x}).
\end{equation}
Fix $\epsilon >0$, by \prettyref{lem:MV-Lemma-7}, there exists an $s_{\epsilon}<\infty$ such that

\begin{multline} \label{eqn:E s epsilon}
  \limsup_{t\to \infty} e^{(n-1-\delta_{\Gamma})t} \int_{\T^l \times \R^{n-1-l} } \chi_{\mathcal{E}_{t,\infty}}(\vect{x},n_+(-\vect{x})a_{-t}) d\lambda(\vect{x}) \\ \le \limsup_{t\to \infty} e^{(n-1-\delta_{\Gamma})t} \int_{\T^l \times \R^{n-1-l} } \chi_{\mathcal{E}_{t,s_{\epsilon}}}(\vect{x},n_+(-\vect{x})a_{-t}) d\lambda(\vect{x})+\epsilon.
\end{multline}
By \prettyref{lem:MV-Lemma-6} for any $\rho = \rho(\epsilon)>0$ there exist sets $\mathcal{A}^{\pm}_{s_{\epsilon},\rho}$, with $\vol(\mathcal{A}^{+}_{s_{\epsilon},\rho}\setminus \mathcal{A}^{-}_{s_{\epsilon},\rho})\le \rho$ and associated

\begin{align}\label{E seps def}
  \mathcal{E}_{s_{\epsilon},\rho}^{\pm} = \mbox{supp}(\lambda) \times \{\alpha \in \Gamma \backslash G : 0 < \#(\alpha^{-1}\overline{\vect{w}}  \cap \mathcal{Z}(s_{\epsilon}, \mathcal{A}^{\pm}_{s_{\epsilon},\rho }))<r \},
\end{align}
such that the right hand side of \eqref{eqn:E s epsilon} is less than

\begin{equation}\label{old 4.44}
   \eqref{eqn:E s epsilon} \le \limsup_{t\to \infty} e^{(n-1-\delta_{\Gamma})t} \int_{\T^l \times \R^{n-1-l} } \chi_{\mathcal{E}_{s_{\epsilon},\rho}^+}(\vect{x},n_+(-\vect{x})a_{-t}) d\lambda(\vect{x})+\epsilon.
 \end{equation}
 Therefore, applying Corollary \ref{cor:test-functions} to \eqref{old 4.44} we can bound
 \begin{align}
   \limsup_{t\to \infty} e^{(n-1-\delta_{\Gamma})t} \int_{\T^l \times \R^{n-1-l} } \chi_{\mathcal{E}_{t,\infty}}(\vect{x},n_+(-\vect{x})a_{-t}) d\lambda(\vect{x}) \le \frac{A_{\lambda}}{|\mathrm{m}^{BMS}|}\mathrm{m}^{BR}(\overline{\mathcal{E}_{s_{\epsilon},\rho}^+})+\epsilon
 \end{align}
 and similarly
  \begin{align}
   \liminf_{t\to \infty} e^{(n-1-\delta_{\Gamma})t} \int_{\T^l \times \R^{n-1-l} } \chi_{\mathcal{E}_{t,\infty}}(\vect{x},n_+(-\vect{x})a_{-t}) d\lambda(\vect{x}) \le \frac{A_{\lambda}}{|\mathrm{m}^{BMS}|}\mathrm{m}^{BR}((\mathcal{E}_{s_{\epsilon},\rho}^-)^o)-\epsilon
 \end{align}

Therefore it remains to use $\rho = \rho(\epsilon)$ to control

\begin{equation}
  \lim_{\epsilon \to 0} \mathrm{m}^{BR} (\overline{\mathcal{E}_{s_{\epsilon},\rho}^+} \setminus (\mathcal{E}_{s_{\epsilon},\rho}^-)^o)
\end{equation}
by Proposition \prettyref{prop:MV-Lemma-5} we have

\begin{equation}
  \lim_{\epsilon \to 0} \mathrm{m}^{BR} (\overline{\mathcal{E}_{s_{\epsilon},\rho}^+} \setminus (\mathcal{E}_{s_{\epsilon},\rho}^-)^o) \le \lim_{\epsilon\to 0}c_{s_{\epsilon},\rho} \vol( \overline{\mathcal{A}_{s_{\epsilon},\rho}^+} \setminus (\mathcal{A}_{s_{\epsilon},\rho}^-)^o) 
\end{equation}
where $c_{s_{\epsilon,\rho}}$ is the constant $C_{t_0}$ defined below \eqref{eqn:thin-Siegel}, here $t_0$ depends on the set $\overline{\mathcal{E}_{s_{\epsilon},\rho}^+}$.

\begin{align}
  t_0 = \inf (\tilde{t} :  0 < \#((n_+a_{-\tilde{t}})^{-1}\overline{\vect{w}}  \cap \mathcal{Z}(\infty,\mathcal{A}^{\pm}_{s_{\epsilon},\rho })) < r_j \forall j)
\end{align}
For fixed $\epsilon$, $\mathcal{Z}(s_{\epsilon},\mathcal{A}^+_{s_{\epsilon},\rho})$ is a cuspidal cone of fixed height. Therefore $t_0$ is bounded below, independent of $\rho>0$. Thus there exists a constant $C^{\prime}_{s_{\epsilon}}$ depending only on $s_{\epsilon}$ such that
\begin{align}
  \lim_{\epsilon \to 0} \mathrm{m}^{BR} (\overline{\mathcal{E}_{s_{\epsilon},\rho}^+} \setminus (\mathcal{E}_{s_{\epsilon},\rho}^-)^o) \le \lim_{\epsilon \to 0} C_{s_{\epsilon}}^\prime \rho(\epsilon)  = 0
\end{align}
for $\rho(\epsilon)$ suitably chosen. Hence

\begin{align}
  \begin{aligned}
    \lim_{\epsilon \to 0} \mathrm{m}^{BR}(\overline{\mathcal{E}_{s_\epsilon,\rho(\epsilon)}^+}) &= \lim_{\epsilon \to 0}\mathrm{m}^{BR}((\mathcal{E}_{s_\epsilon,\rho(\epsilon)}^-)^o)\\
    &= \mathrm{m}^{BR}(\{\alpha \in \Gamma \backslash G : 0 < \#(\alpha^{-1}\overline{\vect{w}} \cap \mathcal{Z}(\infty,\mathcal{A}))\le r_j \forall j \}),
  \end{aligned}
\end{align}
proving the Theorem \ref{thm:joint-limit}.

\end{proof}

  \section{Moment Generating Function for Cuspidal Observer}

Continuing to follow the example set by \cite{MarklofVinogradov2018}, for test sets $\mathcal{A}_1,..., \mathcal{A}_m \subset \R^{n-1}$ with boundary of Lebesgue measure $0$ and for complex $\tau_i \in \C$, define the moment generating function

\begin{align}
  \begin{aligned}
    &\mathbb{G}_{t,s}^{\infty}(\tau_1,...,\tau_m; \mathcal{A}) \\
    &\phantom{+++} := \int_{\T^l \times \R^{n-1-l}} \mathbbm{1}(\mathcal{N}_{t,s}^{\infty}(\mathcal{A}_j,\vect{x}, \overline{\vect{w}}) \neq 0, \forall j)\exp\left(\sum_{j=1}^m \tau_j \mathcal{N}_{t,s}^{\infty}(\mathcal{A}_j,\vect{x}, \overline{\vect{w}})   \right) d\lambda(\vect{x})
    \end{aligned}
\end{align}
and similarly for the limit distribution let

\begin{equation}
  \mathbb{G}_{s}(\tau_1,...,\tau_m; \mathcal{A}) := \sum_{r_1,...,r_m=1}^{\infty} \exp \left(\sum_{j=1}^m \tau_j r_j  \right)E_s(r,\mathcal{A},\overline{\vect{w}}).
\end{equation}
Where $E_s$ is defined as in \prettyref{thm:joint-limit} and $r=(r_1,...,r_m)$. Let $\operatorname{Re}_+ \tau := \max (\operatorname{Re}(\tau),0)$.


\begin{theorem} \label{thm:moment-generating}
  Let $\lambda$ be a Borel probability measure on $\T^l \times \R^{n-1-l}$ as in \prettyref{thm:joint-limit}, and $\{\mathcal{A}\}_{j=1}^m \subset \R^{(n-1)}$ bounded with boundary of Lesbegue measure $0$. Then there exists a constant $c_0>0$ such that for $\operatorname{Re}_+ \tau_1 + ... + \operatorname{Re}_+\tau_m < c_0$, $s \in (0,\infty]$

  \begin{enumerate}
    \item $\mathbb{G}_s(\tau_1,..., \tau_m;\mathcal{A})$ is analytic
    \item $\lim_{t\to \infty}e^{(n-1-\delta_{\Gamma})t} \mathbb{G}_{t,s}^{\infty}(\tau_1,..., \tau_m;\mathcal{A})= \frac{A_{\lambda}}{|\mathrm{m}^{BMS}|}\mathbb{G}_s(\tau_1,..., \tau_m;\mathcal{A})$.
  \end{enumerate}

\end{theorem}

Suppose $-\infty < a< b \le \infty$ and $\mathcal{A} \subset \R^{n-1}$. For $b<\infty$, $\mathcal{Z}(a,b,\mathcal{A})$ (see \eqref{Z(3) def}) is bounded, note that there exists a lattice $\tilde{\Gamma}$ such that $\Gamma < \tilde{\Gamma}$. Hence

\begin{align}
  \begin{aligned}
  \#(\alpha\overline{\vect{w}} \cap \mathcal{Z}(a,b,\mathcal{A})) &\le \#(\alpha\tilde{\Gamma}\vect{w} \cap \mathcal{Z}(a,b,\mathcal{A})) \\
  &\le C \vol((\alpha \mathcal{Z}(a,b,\mathcal{A}))
  \end{aligned}
\end{align}
which, by the left invariance of the volume is unifomly bounded from above in $\alpha$. Thus $\#(\alpha\overline{\vect{w}} \cap \mathcal{Z}(a,b,\mathcal{A}))$ is bounded from above uniformly in $\alpha \in G$. This implies that all moments converge. Therefore we are concerned with the case $b=\infty$.

For that, let 

\begin{equation}
  \delta(\alpha\overline{\vect{w}}):= \min_{\substack{\gamma_1,\gamma_2 \in \Gamma\\  \gamma_1 \not\in \gamma_2\Gamma_{\vect{w}}}} d(\alpha\gamma_1\vect{w}, \alpha\gamma_2\vect{w}).
\end{equation}
Note, because $\alpha$ is an isometry and because $G$ acts properly discontinuously

\begin{equation}
  \delta(\alpha\overline{\vect{w}}) = \min_{\gamma \in  \Gamma/\Gamma_{\vect{w}}} d(\vect{w},\gamma\vect{w} )=\delta(\overline{\vect{w}})>0.
\end{equation}

In order to prove \prettyref{thm:moment-generating} we first require three lemmas.


\begin{lemma} \label{lem:MV-Lemma-10}
  Fix $a\in \R$ and a bounded subset $\mathcal{A} \subset \R^{n-1}$. There exist positive constant $\zeta, \eta$ such that for all $\alpha \in G$, $r\in \N$

  \begin{equation}
    \left[ \#(\alpha\overline{\vect{w}} \cap \mathcal{Z}(a,\infty, \mathcal{A}))\ge r \right ] \Rightarrow \left[ \#(\alpha\overline{\vect{w}} \cap \mathcal{Z}(\zeta r - \eta,\infty, \mathcal{A}))\ge 1 \right]
  \end{equation}

\end{lemma}

This lemma is a statement about the definition of $\mathcal{Z}$. As the definition is the same as in \cite{MarklofVinogradov2018} we do not include the proof (see {\cite[Lemma 10]{MarklofVinogradov2018}}).


\begin{lemma}\label{lem:MV-Lemma-11}
  Fix a bounded subset $\mathcal{A} \subset \R^{n-1}$ and $\zeta, \eta $ as in \prettyref{lem:MV-Lemma-10}. Then

 \begin{eqnarray}
    \int_{\Gamma \backslash G} \# (\alpha^{-1}\overline{\vect{w}} \cap \mathcal{Z}(\zeta r - \eta, \infty,\mathcal{A}))d\mathrm{m}^{BR}(\alpha) \le  \frac{e^{(\eta -\zeta r)}\vartheta^{(n-1)/\delta_{\Gamma}}\vol_{\half^n}(\mathcal{A})}{\#\Gamma_{\vect{w}}(n-1)}
  \end{eqnarray}

\end{lemma}

\begin{proof}\phantom{\qedhere}
  
  This statement follows quite straightfowardly from Proposition \prettyref{prop:MV-Lemma-5} and specifically (\ref{eqn:thin-Siegel}). To see this note

  \begin{align}
    \int_{\Gamma \backslash G} \# (\alpha^{-1}\overline{\vect{w}} \cap \mathcal{Z}(\zeta r - \eta, \infty,\mathcal{A}))d\mathrm{m}^{BR}(\alpha) = \frac{1}{\#\Gamma_{\vect{w}}} \int_G \chi_{\mathcal{Z}(\zeta r-\eta , \infty, \mathcal{A})}( \alpha^{-1}\vect{w}) d\mathrm{m}^{BR}(\alpha) \label{eqn:LM11-eqn2}
  \end{align}
  Now if we apply \eqref{eqn:thin-Siegel} and then insert the volume of $\mathcal{Z}(\zeta r - \eta, \infty, \mathcal{A})$:
  \begin{align}
    \begin{aligned}
      &\int_{\Gamma \backslash G} \# (\alpha^{-1}\overline{\vect{w}} \cap \mathcal{Z}(\zeta r - \eta, \infty,\mathcal{A}))d\mathrm{m}^{BR}(\alpha)\\
      &\phantom{++++++++++++++++}\le \frac{e^{(-n+1 -\delta_{\Gamma})(\zeta r - \eta)}}{\#\Gamma_{\vect{w}}}\vol_{\half^n}(g^{-1}_{\vect{w}} \mathcal{Z}(\zeta r - \eta, \infty, \mathcal{A})) \\
    &\phantom{+++++++++++++++}= \frac{e^{\delta_{\Gamma}(\eta -\zeta r )}\vartheta^{(n-1)/\delta_{\Gamma}}\vol_{\half^n}(\mathcal{A})}{\#\Gamma_{\vect{w}}(n-1)} \label{eqn:LM11-eqn4}.
    \end{aligned}
  \end{align}

\end{proof}


\begin{lemma}\label{lem:MV-Lemma-12}
  Fix a bounded subset $\mathcal{A} \subset \R^{n-1}$ and $\zeta,\eta$ as in \prettyref{lem:MV-Lemma-10}. Let $\lambda$ be a probability measure on $\T^l \times \R^{n-1-l}$ as in \prettyref{thm:joint-limit}. Then, there exists a constant $C$ such that

  \begin{equation}\label{Lemma 5.4}
    \sup_{t \ge 0}e^{(n-1-\delta_\Gamma)t}\int_{\T^l \times \R^{n-1-l}} \# (a_t n_+(\vect{x})\overline{\vect{w}} \cap \mathcal{Z}(\zeta r - \eta, \infty, \mathcal{A})) d \lambda(\vect{x}) \le Ce^{-\zeta r \delta_{\Gamma}}.
   \end{equation}
\end{lemma}

\begin{proof}\phantom{\qedhere}

  The proof follows very similar lines to {\cite[Lemma 12]{MarklofVinogradov2018}}. Firstly by taking $C$ large we may assume $\lambda$ is the Lesbegue measure on the support of $\lambda$. Then

  \begin{align}
    \begin{aligned}
    &\int_{\T^l \times \R^{n-1-l}} \# (a_tn_+(\vect{x})\overline{\vect{w}} \cap \mathcal{Z}(\zeta r - \eta, \infty, \mathcal{A}))\chi_{\mbox{supp}(\lambda)}(\vect{x})  d\vect{x} \\ 
    &\phantom{============}= \int_{\T^l \times \R^{n-1-l}} \# (n_+(\vect{x})\overline{\vect{w}} \cap \mathcal{Z}(\zeta r - \eta-t, \infty, e^{-t}\mathcal{A})) \chi_{\mbox{supp}(\lambda)}(\vect{x}) d\vect{x}\\
    &\phantom{============}\le C \vol_{\R^{n-1}}(e^{-t}\mathcal{A}) \# \{ \gamma \in \Gamma_\infty\backslash \Gamma / \Gamma_{\vect{w}} , \operatorname{Im}(\gamma \vect{w}) \ge e^{-t+\zeta r - \eta}\}.
    \end{aligned}
  \end{align}
  By (\ref{eqn:P-asymptotics}) there exists a constant such that

  \begin{equation}
    \# \{ \gamma \in \Gamma_\infty\backslash \Gamma / \Gamma_{\vect{w}}, \operatorname{Im}(\gamma\vect{w}) \ge e^{-t+\zeta r - \eta}\} \le C' \max\{1,e^{-\delta_{\Gamma}(t-\zeta r)}\}.
  \end{equation}
  from which \eqref{Lemma 5.4} follows.
 
\end{proof}

\begin{proof}\phantom{\qedhere}[Proof of \prettyref{thm:moment-generating}]
  To begin with we once more note that for $s<\infty$, $\mathcal{N}_{t,s}^{\infty}(\mathcal{A}_j, \vect{x}; \overline{\vect{w}})$ is uniformly bounded and thus $E_s(r,\mathcal{A}; \overline{\vect{w}}) = 0$ for $|r|:=\max_j r_j$ large enough. From here Theorem \ref{thm:moment-generating} follows. Thus we set $s=\infty$ for the remainder of the proof.
  
  Set $\tilde{\mathcal{A}} = \bigcup_j \mathcal{A}_j$

  \begin{eqnarray}
    \sum_{|r| \ge R} E_s(r,\mathcal{A}; \overline{\vect{w}}) &\le& \sum_{r=R}^{\infty}E_s(r,\tilde{\mathcal{A}}; \overline{\vect{w}})\\
                     &\le& \frac{A_{\lambda}}{|\mathrm{m}^{BMS}|}\mathrm{m}^{BR}(\{ \alpha \in \Gamma \backslash G : \#(\alpha^{-1}\overline{\vect{w}}  \cap \mathcal{Z}(0,\infty, \tilde{\mathcal{A}}) \ge R\}) \nonumber\\
                     &\le& \frac{A_{\lambda}}{|\mathrm{m}^{BMS}|} \mathrm{m}^{BR}(\{ \alpha \in \Gamma \backslash G : \#(\alpha^{-1}\overline{\vect{w}}  \cap \mathcal{Z}(\zeta R - \eta,\infty, \tilde{\mathcal{A}}) \ge 1\}) \nonumber
  \end{eqnarray}
  where we have used \prettyref{lem:MV-Lemma-10}. Now by Chebyshev's inequality,

  \begin{align}
    \sum_{|r| \ge R} E_s(r,\mathcal{A}; \overline{\vect{w}}) \le  \frac{A_{\lambda}}{|\mathrm{m}^{BMS}|} \int_{\Gamma \backslash G} \# (\alpha^{-1}\overline{\vect{w}} \cap \mathcal{Z}( \zeta R - \eta, \infty, \tilde{\mathcal{A}})) d\mathrm{m}^{BR}(\alpha).
  \end{align}
  We can then use \prettyref{lem:MV-Lemma-11} to say

  \begin{equation}
    \sum_{|r| \ge R} E_s(r,\mathcal{A};\overline{\vect{w}}) \le C_1 e^{-\delta_{\Gamma}\zeta R}
  \end{equation}
  from which analyticity follows. 

  \prettyref{thm:joint-limit} implies

  \begin{multline}
    \lim_{t\to \infty} e^{(n-1-\delta_{\Gamma})t}\int_{\T^l \times \R^{n-1-l}}\prod_{j=1}^m \mathbbm{1}(0<\mathcal{N}_{t,s}^{\infty}(\mathcal{A}_j,\vect{x};\overline{\vect{w}}) <R) \exp (\tau_j \mathcal{N}_{t,s}^{\infty}(\mathcal{A}_j,\vect{x}; \overline{\vect{w}}))d\lambda(\vect{x}) \\
    = \frac{A_{\lambda}}{|\mathrm{m}^{BMS}|}\sum_{r_1,...,r_m=1}^{R-1}\exp\left( \sum_{j=1}^m \tau_j r_j\right) E_s(r,\mathcal{A}; \overline{\vect{w}}).
  \end{multline}
  Therefore it remains to show
  
  \begin{align}
    \begin{aligned}
      &\lim_{R \to \infty}\limsup_{t\to \infty}e^{(n-1-\delta_{\Gamma})t} \left| \int_{\T^l \times \R^{n-1-l}}\prod_{j=1}^m \mathbbm{1}(\max_j\mathcal{N}_{t,s}^{\infty}(\mathcal{A}_j,\vect{x};\overline{\vect{w}}) \ge R ) \cdot\right.\\
      &\phantom{++++++++++++}\left.\mathbbm{1}(\min_j\mathcal{N}_{t,s}^{\infty}(\mathcal{A}_j,\vect{x};\overline{\vect{w}}) >0 )\exp (\tau_j \mathcal{N}_{t,s}^{\infty}(\mathcal{A}_j,\vect{x}; \overline{\vect{w}}))d\lambda(\vect{x})\right| =0.
    \end{aligned}
  \end{align}
  Note that

  \begin{multline}\label{(5.21)}
    e^{(n-1-\delta_{\Gamma})t}\left| \int_{\T^l \times \R^{n-1-l}}\prod_{j=1}^m \mathbbm{1}(\max_j\mathcal{N}_{t,s}^{\infty}(\mathcal{A}_j,\vect{x};\overline{\vect{w}}) \ge R) \exp (\tau_j \mathcal{N}_{t,s}^{\infty}(\mathcal{A}_j,\vect{x}; \overline{\vect{w}}))d\lambda(\vect{x})\right| \\\le e^{(n-1-\delta_{\Gamma})t} \int_{\T^l \times \R^{n-1-l}}\mathbbm{1}(\mathcal{N}_{t,s}^{\infty}(\mathcal{A},\vect{x}; \overline{\vect{w}}) \ge R)\exp (\tilde{\tau}\mathcal{N}_{t,s}^{\infty}(\tilde{\mathcal{A}},\vect{x};\overline{\vect{w}}))d\lambda(\vect{x}),
  \end{multline}
  where $\tilde{\mathcal{A}}= \bigcup_j \mathcal{A}_j$ and $\tilde{\tau} = \sum_j \operatorname{Re}_+\tau_j$. From there, performing the same decomposition as {\cite[proof of Theorem 8]{MarklofVinogradov2018}} we get that the right hand side of \eqref{(5.21)} is less than or equal

  \begin{equation}\label{(5.22)}
    \eqref{(5.21)} \le e^{(n-1-\delta_{\Gamma})t}\sum_{r=R}^{\infty} e^{\tilde{\tau}r}\int_{\T^l \times \R^{n-1-l}}\mathbbm{1}(\mathcal{N}_{t,s}^{\infty}(\tilde{\mathcal{A}},\vect{x};\overline{\vect{w}})\ge r )d\lambda(\vect{x}).
  \end{equation}
  Now using \prettyref{lem:MV-Lemma-10} and \prettyref{lem:MV-Lemma-12} we can bound \eqref{(5.22)} (uniformly in $t \ge 0$) by

  \begin{equation}
    \eqref{(5.22)} \le \sum_{r=R}^{\infty} C e^{\tilde{\tau}r}e^{-\delta_{\Gamma} \zeta r}.
  \end{equation}
  Thus, for $\tilde{\tau} < \delta_{\Gamma} \zeta$

  \begin{equation}
    \lim_{R \to 1} e^{(n-1-\delta_{\Gamma})t} \sum_{r=R}^{\infty} e^{\tilde{\tau}r}\int_{\T^l \times \R^{n-1-l}}\mathbbm{1}(\mathcal{N}_{t,s}^{\infty}(\tilde{\mathcal{A}},\vect{x};\overline{\vect{w}})\ge r )d\lambda(\vect{x}) = 0
  \end{equation}
  uniformly in $t$. Taking $c_0=\delta_{\Gamma} \zeta$ proves Theorem \ref{thm:moment-generating}.

\end{proof}

  \section{Spherical Averages}

We now present a Theorem analogous to Theorem \ref{thm:multiple-functions} however we will replace the horospherical average with a spherical average. This will allow us to move the observer to the interior and replace the shrinking horospherical subset with a shrinking subset of the sphere centered on the observer. Fix $g\in G$ and recall the definition of the spherical Patterson-Sullivan measure, $\mu^{PS}_{\Gamma g\overline{K}}$ - \eqref{spherical PS}. Moreover, given a subset $\mathcal{U} \subset \R^{n-1}$ and parameterisation $R: \vect{x} \to \overline{K}$ from $\mathcal{U}$, as in Subsection 2.2, recall the definition of $\omega^{PS}_{\Gamma,g,\overline{K}}$. Finally we use the notation $R(\vect{x})^{-1}$ to denote the inverse matrix.


\begin{theorem} \label{thm:spherical-averages}
  Let $\mathcal{U}$ be a nonempty open subset and let $R:\mathcal{U}\to \overline{K}$ such that the map $\mathcal{U} \ni \vect{x} \mapsto R(\vect{x})^{-1}\vect{0} \in \partial \half^n$ has nonsingular differential at almost all $\vect{x} \in \mathcal{U}$. Let $\lambda$ be a compactly supported Borel probability measure on $\mathcal{U}$ with continuous density. Then for any compactly supported, right $M$-invariant, continuous $f: \mathcal{U} \times \Gamma \backslash G \to \R$, and any family of right $M$-invariant, continuous $f_t: \mathcal{U} \times \Gamma \backslash G \to \R$ all supported on a single compact set, with $f_t \to f$ as $t \to \infty$ uniformly, for any $g\in G$

  \begin{align}
    \begin{aligned}
      &\lim_{t \to \infty} e^{(n-1-\delta_{\Gamma})t}\int_{\mathcal{U}} f_t (\vect{x}, \Gamma g R(\vect{x})a_t) d \lambda(\vect{x})\\
      &\phantom{+++++++++++}= \frac{1}{|\mathrm{m}^{BMS}|} \int_{ \mathcal{U} \times \Gamma \backslash G }\lambda'(\vect{x}) f(\vect{x},\alpha) d\mathrm{m}^{BR}(\alpha) d\omega^{PS}_{\Gamma,g,\overline{K}}(\vect{x}).  \label{eqn:first-spherical-average}
    \end{aligned}
  \end{align}

\end{theorem}

\begin{proof}\phantom{\qedhere}
The proof follows the same lines as {\cite[Corollary 5.4]{MarklofStrom2010}} but requires some significant additions due to the invariance of the limiting measure. 

Let $\vect{x}_0$ be a point where the map $\vect{x} \mapsto R(\vect{x})^{-1}\vect{0}$ has non-singular differential. We first show that (\ref{eqn:first-spherical-average}) holds for any Borel subset of an open set $\mathcal{U}_0 \subset \mathcal{U}$ containing $\vect{x}_0$. As $R(\vect{x}) \in \overline{K}$ we can write

\begin{equation}
  R(\vect{x}) = \mat{\vect{a}(\vect{x})}{\vect{b}(\vect{x})}{-\vect{b}^\prime(\vect{x})}{\vect{a}^\prime(\vect{x})}
\end{equation}
where $\vect{a}(\vect{x}), \vect{b}(\vect{x}) \in \Delta_{n-2}$.

\emph{Case 1:} Assume $\vect{a}(\vect{x}_0)\neq 0$. In that case we write

  \begin{align*}
    R(\vect{x}) & = \mat{\vect{a}(\vect{x})}{\vect{b}(\vect{x})}{-\vect{b}^\prime(\vect{x})}{\vect{a}^\prime(\vect{x})}\\
                & = \mat{1}{0}{-\vect{b}^\prime(\vect{x})\vect{a}(\vect{x})^{-1}}{1} \mat{\vect{a}(\vect{x})}{\vect{b}(\vect{x})}{0}{\vect{b}^\prime(\vect{x})\vect{a}(\vect{x})^{-1}\vect{b}(\vect{x}) + \vect{a}^\prime(\vect{x})}\\
                &  = \mat{1}{0}{\tilde{\vect{x}}}{1} \mat{\vect{a}(\vect{x})}{\vect{b}(\vect{x})}{0}{-\tilde{\vect{x}}\vect{b}(\vect{x})+\vect{a}'(\vect{x})},
  \end{align*}
  with $\tilde{\vect{x}} : = - \vect{b}^\prime (\vect{x}) \vect{a}^{-1}(\vect{x}) = R(\vect{x})^{-1}\vect{0}$. Note further that

  \begin{eqnarray}
    R(\vect{x})a_t &=& \mat{1}{0}{\tilde{\vect{x}}}{1} a_t \mat{\vect{a}(\vect{x})}{e^{-t}\vect{b}(\vect{x})}{0}{-\tilde{\vect{x}}\vect{b}(\vect{x})+\vect{a}'(\vect{x})} \\
     &=& n_-(\tilde{\vect{x}})a_t \mat{\vect{a}(\vect{x})}{e^{-t}\vect{b}(\vect{x})}{0}{-\tilde{\vect{x}}\vect{b}(\vect{x})+\vect{a}'(\vect{x})}.
  \end{eqnarray}
   As the map $\vect{x} \mapsto \vect{x}_0$ has nonsingular differential at $\vect{x}_0$ there exists an open set $\mathcal{V} \ni \vect{x}_0$ such that $\overline{\mathcal{V}} \subset \mathcal{U}$ and $\vect{x} \mapsto \vect{x}_0$ is a diffeomorphism on $\mathcal{V}$. We call the image under this map $\tilde{\mathcal{V}}$ (and adopt this notation for all subsets of $\mathcal{V}$).

  Let $\mathcal{U}_0$ be an open neighborhood of $\vect{x}_0$ such that $\overline{\mathcal{U}}_0 \subset \mathcal{V}$. For any Borel subset $B \subset \mathcal{U}_0$ we have

\begin{equation}
  \tilde{B} \subset \tilde{\mathcal{U}}_0 \subset \tilde{\mathcal{V}}.
\end{equation}
Assume $\lambda(B)>0$ and let $\tilde{\lambda}$ be the push-forward measure on $\R^{n-1}$ of $\frac{1}{\lambda(B)}\lambda\left. \right|_{B}$ by the map $\vect{x} \mapsto \tilde{\vect{x}}$. Note $\tilde{\lambda}$ has compact support and continuous density. 

Let $u$ be a continuous function with $\chi_{\tilde{\mathcal{U}}_0} \le u \le \chi_{\tilde{\mathcal{V}}}$. With that let $\tilde{f}_t,\tilde{f} :\R^{n-1} \times \Gamma \backslash G \to \R$ be the continuous and compactly support functions

\begin{align}
  \tilde{f}_t(\tilde{\vect{x}},\alpha) &= u(\tilde{\vect{x}})f_t\left(\vect{x}, \alpha \mat{\vect{a}}{e^{-t}\vect{b}}{0}{-\tilde{\vect{x}}\vect{b}+\vect{a}'}\right),   & \tilde{\vect{x}}\in \tilde{\mathcal{V}}\\
  \tilde{f}(\tilde{\vect{x}},\alpha) &= u(\tilde{\vect{x}})f\left(\vect{x}, \alpha \mat{\vect{a}}{0}{0}{-\tilde{\vect{x}}\vect{b}+\vect{a}'}\right),  & \tilde{\vect{x}}\in \tilde{\mathcal{V}}\\
  \tilde{f}_t(\tilde{\vect{x}},\alpha)&=\tilde{f}(\tilde{\vect{x}},\alpha) =0, &\tilde{\vect{x}}\not\in \tilde{\mathcal{V}}.
\end{align}
With that, we can apply \prettyref{thm:multiple-functions} to $\tilde{f}_t$,

\begin{align} \label{eqn:u final}
  \begin{aligned}
    &\lim_{t\to \infty} e^{(n-1-\delta_{\Gamma})t} \int_{\R^{n-1}} u(\tilde{\vect{x}})f_t(\vect{x},\Gamma g R(\vect{x})a_t ) d\lambda(\vect{x})\\
    &\phantom{++++++++} = \lim_{t\to \infty} e^{(n-1-\delta_{\Gamma})t} \int_{\R^{n-1}} \tilde{f}_t(\tilde{\vect{x}},\Gamma g n_-(\tilde{\vect{x}})a_t ) d\tilde{\lambda}(\tilde{\vect{x}}) \\
    &\phantom{++++++++}=\frac{1}{|\mathrm{m}^{BMS}|}\int_{ \R^{n-1} \times \Gamma \backslash G } \tilde{\lambda}'(\tilde{\vect{x}}) \tilde{f}(\tilde{\vect{x}},\alpha) d\mathrm{m}^{BR}(\alpha) d\omega^{PS}_{\Gamma,g,\overline{H}}(\tilde{\vect{x}}).
 \end{aligned}
\end{align}
To complete the proof we have the following claim

\emph{Claim:}

\begin{align} \label{eqn:Claim}
  \begin{aligned}
    &\int_{\R^{n-1} \times \Gamma \backslash G } \tilde{\lambda}'(\tilde{\vect{x}}) \tilde{f}(\tilde{\vect{x}},\alpha) d\mathrm{m}^{BR}(\alpha) d\omega^{PS}_{\Gamma,g,\overline{H}}(\tilde{\vect{x}})\\
    &\phantom{+++++++++++++}= \int_{ \mathcal{U} \times \Gamma \backslash G }u(\tilde{\vect{x}})\lambda'(\vect{x}) f(\vect{x},\alpha) d\mathrm{m}^{BR}(\alpha) d\omega^{PS}_{\Gamma,g,\overline{K}}(\vect{x}).
  \end{aligned}
\end{align}

Accepting the claim for the moment, we have proved the Theorem \ref{thm:spherical-averages} for a Borel subset $B\subset \mathcal{U}_0$. The full Theorem \ref{thm:spherical-averages} follows in this case by a covering argument which is the same as the one presented in {\cite[Corollary 5.4]{MarklofStrom2010}}.

\emph{Case 2:} If $\vect{a}(\vect{x}_0) =0$, then we can write 

\begin{equation}
  R(\vect{x}) = \mat{\vect{a}}{\vect{b}}{-\vect{b}'}{\vect{a}'} = \mat{0}{1}{-1}{0} \mat{\vect{b}'}{-\vect{a}'}{\vect{a}}{\vect{b}} =: \mat{0}{1}{-1}{0}R_0(\vect{x})
\end{equation}
where $\vect{b}(\vect{x}_0) \neq 0$. Thus we can replace $g$ in (\ref{eqn:u final}) with $g \mat{0}{1}{-1}{0}$. From here the proof follows the same lines as Case 1.

\emph{Proof of Claim:}

\textbf{Step 1:}

Expanding the left hand side of \eqref{eqn:Claim}

\begin{multline}\label{expanded lhs}
  \int_{\R^{n-1} \times \Gamma \backslash G } \tilde{\lambda}'(\tilde{\vect{x}}) \tilde{f}(\tilde{\vect{x}},\alpha) d\mathrm{m}^{BR}(\alpha) d\omega^{PS}_{\Gamma,g,\overline{H}}(\tilde{\vect{x}}) \\= \int_{\R^{n-1} \times \Gamma \backslash G } \tilde{\lambda}'(\tilde{\vect{x}})u(\tilde{\vect{x}}) f(\vect{x},\alpha \begin{psmallmatrix} \vect{a} & 0 \\ 0 & -\tilde{\vect{x}}\vect{b}+\vect{a}^\prime \end{psmallmatrix}) d\mathrm{m}^{BR}(\alpha) d\omega^{PS}_{\Gamma,g,\overline{H}}(\tilde{\vect{x}}).
\end{multline}
We may write $ \begin{psmallmatrix} \vect{a} & 0 \\ 0 & -\tilde{\vect{x}}\vect{b}+\vect{a}^\prime \end{psmallmatrix} = \begin{psmallmatrix} |\vect{a}(\vect{x})| & 0 \\ 0 & |\vect{a}(\vect{x})|^{-1} \end{psmallmatrix} M(\vect{x})$ where $M(\vect{x}) \in M$. Since $f$ is right $M$-invariant, $M(\vect{x})$ can be ignored. Now note that the Burger-Roblin measure is 'quasi-invariant' for the geodesic flow (see {\cite[(2)]{Mohammadi2013}}) thus

\begin{multline} \label{Step 1}
  \int_{ \R^{n-1} \times \Gamma \backslash G } \tilde{\lambda}'(\tilde{\vect{x}}) \tilde{f}(\tilde{\vect{x}},\alpha) d\mathrm{m}^{BR}(\alpha) d\omega^{PS}_{\Gamma,g,\overline{H}}(\tilde{\vect{x}}) \\= \int_{ \R^{n-1} \times \Gamma \backslash G } |\vect{a}(\vect{x})|^{(n-1-\delta_{\Gamma})} \tilde{\lambda}'(\tilde{\vect{x}})u(\tilde{\vect{x}}) f(\vect{x},\alpha) d\mathrm{m}^{BR}(\alpha) d\omega^{PS}_{\Gamma,g,\overline{H}}(\tilde{\vect{x}}).
\end{multline}

\textbf{Step 2:}

First we note that since $R(\vect{x}) \in \overline{K}$,  $\vect{a}\overline{\vect{a}} + \vect{b}\overline{\vect{b}} = 1$ and thus

\begin{align*}
  R(\vect{x}) & = n_-(\tilde{\vect{x}}) \mat{\vect{a}(\vect{x})}{\vect{b}(\vect{x})}{0}{\vect{a}^{\ast}(\vect{x})^{-1}}
\end{align*}
Which we can further decompose

\begin{align} \begin{aligned} \label{R decomp}
    R(\vect{x}) & = n_-(\tilde{\vect{x}}) n_+(\vect{b}(\vect{x})\vect{a}^{\ast}(\vect{x})) \mat{|\vect{a}(\vect{x})|}{0}{0}{|\vect{a}(\vect{x})|^{-1}} \mat{\frac{\vect{a}(\vect{x})}{|\vect{a}(\vect{x})|}}{0}{0}{\frac{\vect{a}^{\ast}(\vect{x})^{-1}}{|\vect{a}(\vect{x})|^{-1}}},\\
    & = n_-(\tilde{\vect{x}}) A(\vect{x}),
    \end{aligned}
\end{align}
where we have defined $ A(\vect{x}) : = n_+(\vect{b}(\vect{x})\vect{a}^{\ast}(\vect{x})) \begin{psmallmatrix}|\vect{a}(\vect{x})| & 0 \\0 &|\vect{a}(\vect{x})|^{-1} \end{psmallmatrix}\begin{psmallmatrix}  \frac{\vect{a}(\vect{x})}{|\vect{a}(\vect{x})|}  &  0  \\  0 & \frac{\vect{a}^{\ast}(\vect{x})^{-1} }{ |\vect{a}(\vect{x})|^{-1}}\end{psmallmatrix}$.  Note that the last matrix is in $M$. As we are working on $\overline{K}$ this last matrix can be ignored.

Now observe that by using \eqref{R decomp}

\begin{align*}
  gR(\vect{x})\vect{X}_{\vect{i}}^+ & = \lim_{t\to \infty} g R(\vect{x}) a_t \vect{X}_{\vect{i}}\\
                                  & = g n_-(\tilde{\vect{x}}) \vect{X}_{\vect{i}}^+.
\end{align*}
Therefore using the definition of $\omega_{\Gamma,g,\overline{H}}^{PS}$ \eqref{omega def} we can write

\begin{align}\begin{aligned} \label{eqn:PS measure H to K}
  d\omega_{\Gamma,g,\overline{H}}^{PS}(\tilde{\vect{x}}) &= d\mu^{PS}_{\Gamma g \overline{H}}(gn_-(\tilde{\vect{x}}))\\
  &=e^{\delta_{\Gamma} \beta_{gn(\tilde{\vect{x}})\vect{X}_{\vect{i}}^+}(\vect{i}, gn_-(\tilde{\vect{x}})\vect{i})} d\nu_{\vect{i}}(gn_-(\tilde{\vect{x}})\vect{X}_{\vect{i}}^+) \\
  &=e^{\delta_{\Gamma} \beta_{gR(\vect{x})\vect{X}_{\vect{i}}^+}(\vect{i}, gR(\vect{x})A(\vect{x})^{-1}\vect{i})} d\nu_{\vect{i}}(gR(\vect{x})\vect{X}_{\vect{i}}^+) .
\end{aligned}
\end{align}
Note that the Busemann function is both $M$ and $N_+$ invariant (via right multiplication). Hence

\begin{eqnarray}
  \beta_{gR(\vect{x})\vect{X}_{\vect{i}}^+}(\vect{i}, gR(\vect{x})A(\vect{x})^{-1}\vect{i}) &=& \beta_{gR(\vect{x}) \vect{X}_{\vect{i}}^+}\left(\vect{i}, gR(\vect{x}) \mat{|\vect{a}|^{-1}}{0}{0}{|\vect{a}|}\vect{i} \right)\notag \\
  &=& \ln |\vect{a}| + \beta_{gR(\vect{x})\vect{X}_{\vect{i}}^+}\left(\vect{i}, gR(\vect{x})\vect{i}\right)
\end{eqnarray}
Therefore

\begin{align} \label{Step 2}
d \omega^{PS}_{\Gamma,g,\overline{H}}(\tilde{\vect{x}}) = |\vect{a}(\vect{x})|^{\delta_{\Gamma}} e^{\delta_{\Gamma}\beta_{g R(\vect{x}) \vect{X}_{\vect{i}}^+ }(\vect{i},g R(\vect{x})\vect{i})} d\nu_{\vect{i}}(g R(\vect{x}) \vect{X}_{\vect{i}}^+)
\end{align}

\textbf{Step 3:}

Inserting $\tilde{\lambda}^\prime(\tilde{\vect{x}}) = \left| \frac{\partial\vect{\tilde{\vect{x}}}}{\partial\vect{x}}\right|^{-1} \lambda^\prime(\vect{x})$ into \eqref{Step 1} gives

\begin{multline} \label{Step 3+1}
  \int_{ \R^{n-1} \times \Gamma \backslash G} \tilde{\lambda}'(\tilde{\vect{x}}) \tilde{f}(\tilde{\vect{x}},\alpha) d\mathrm{m}^{BR}(\alpha) d\omega^{PS}_{\Gamma,g,\overline{H}}(\tilde{\vect{x}}) \\= \int_{\R^{n-1} \times \Gamma \backslash G } |\vect{a}(\vect{x})|^{(n-1-\delta_{\Gamma})}\left|\frac{\partial \tilde{\vect{x}}}{\partial \vect{x}}\right|^{-1} \lambda'(\vect{x})u(\tilde{\vect{x}}) f(\vect{x},\alpha) d\mathrm{m}^{BR}(\alpha) d\omega^{PS}_{\Gamma,g,\overline{H}}(\tilde{\vect{x}}).
\end{multline}
Now if we insert \eqref{Step 2} into \eqref{Step 3+1} we obtain

\begin{align}
  \begin{aligned}
    &\int_{\R^{n-1}\times \Gamma \backslash G } \tilde{\lambda}'(\tilde{\vect{x}}) \tilde{f}(\tilde{\vect{x}},\alpha) d\mathrm{m}^{BR}(\alpha) d\omega^{PS}_{\Gamma,g,\overline{H}}(\tilde{\vect{x}}) \\
    &\phantom{++++++++++}= \int_{\R^{n-1} \times \Gamma \backslash G }  \lambda'(\vect{x})u(\tilde{\vect{x}}) f(\vect{x},\alpha) d\mathrm{m}^{BR}(\alpha)\cdot \\
    &\phantom{+++++++++++++}\left(\left|\frac{\partial \tilde{\vect{x}}}{\partial \vect{x}}\right|^{-1} |\vect{a}(\vect{x})|^{n-1}e^{\delta_{\Gamma}\beta_{g R(\vect{x}) \vect{X}_{\vect{i}}^+ }(\vect{i},g R(\vect{x})\vect{i})} d\nu_{\vect{i}}(g R(\vect{x}) \vect{X}_{\vect{i}}^+)\right).
    \end{aligned}
\end{align}
Note that the final measure in the brackets is exactly the definition of $d\omega^{PS}_{\Gamma,g,\overline{K}}(\vect{x})$, \eqref{omega PS K def}. Proving the claim.

\end{proof}

 We can extend \prettyref{thm:spherical-averages} to sequences of characteristic functions in much the same way as for Corollary \prettyref{cor:test-functions}


\begin{corollary} \label{cor:test-functions-sphere}
  Under the assumptions of \prettyref{thm:spherical-averages}, for any $g \in \Gamma\backslash G$ and any bounded family of subsets $\mathcal{E}_t \subset \mathcal{U} \times \Gamma \backslash G$ with boundary of $ \omega^{PS}_{\Gamma, g, \overline{K}} \times \mathrm{m}^{BR}$-measure $0$ 

  \begin{multline}
    \liminf_{t \to \infty} e^{(n-1-\delta_{\Gamma})t}\int_{\mathcal{U}} \chi_{\mathcal{E}_t} (\vect{x} , \Gamma g R(\vect{x}) a_t) d\lambda(\vect{x}) \ge \\\frac{1}{|\mathrm{m}^{BMS}|}\int_{ \mathcal{U} \times \Gamma \backslash G } \lambda'(\vect{x}) \chi_{\lim ( \inf\mathcal{E}_t)^o}(\vect{x},\alpha) d\mathrm{m}^{BR}(\alpha) d\omega^{PS}_{\Gamma, g, \overline{K}}(\vect{x})
  \end{multline}
  and

  \begin{multline}
    \limsup_{t \to \infty}e^{(n-1-\delta_{\Gamma})t} \int_{\mathcal{U}} \chi_{\mathcal{E}_t} (\vect{x} , \Gamma g R(\vect{x}) a_t) d\lambda(\vect{x}) \le \\\frac{1}{|\mathrm{m}^{BMS}|}\int_{\mathcal{U} \times \Gamma \backslash G } \lambda'(\vect{x}) \chi_{\lim\overline{\sup \mathcal{E}_t}}(\vect{x},\alpha) d\mathrm{m}^{BR}(\alpha) d\omega^{PS}_{\Gamma, g, \overline{K}}(\vect{x})
  \end{multline}
  If furthermore $\lambda \times \mathrm{m}^{BR}$ gives zero measure to $\lim \overline{\sup \mathcal{E}_t} \backslash \lim ( \inf\mathcal{E}_t)^o$

  \begin{multline}
    \lim_{t \to \infty} e^{(n-1-\delta_{\Gamma})t} \int_{\mathcal{U}} \chi_{\mathcal{E}_t} (\vect{x} , \Gamma g R(\vect{x}) a_t) d\lambda(\vect{x}) = \\\frac{1}{|\mathrm{m}^{BMS}|}\int_{\mathcal{U} \times \Gamma \backslash G } \lambda'(\vect{x}) \chi_{\lim\sup \mathcal{E}_t}(\vect{x},\alpha) d\mathrm{m}^{BR}(\alpha) d\omega^{PS}_{\Gamma, g, \overline{K}}(\vect{x})
  \end{multline}
\end{corollary}

  \section{Projection Statistics for Observers in $\half^n$}

Define the coordinate chart of a neighborhood of the south pole of $S_1^{n-1}$ in $ \half^n$ given by the map

\begin{equation}
  \vect{x} \mapsto E(\vect{x})^{-1}(e^{-1} \vect{i})
\end{equation}
where

\begin{equation}
  E(\vect{x}) =  \left( \exp \mat{0}{\vect{x}}{-\vect{x}'}{0}\right)
\end{equation}
Note that by {\cite[(6.3)]{MarklofVinogradov2018}} the map $\vect{x} \mapsto \tilde{\vect{x}} = E(\vect{x})^{-1}\vect{0} $ has a nonsingular differential for all $|\vect{x}| < \pi/2$ hence we can apply Corollary \prettyref{cor:test-functions-sphere}.

Define the shrinking test set

\begin{equation}
  \mathcal{B}_{t,s}(\mathcal{A},0) := \{ E(\vect{x})^{-1}(e^{-1}\vect{i}): \vect{x} \in \rho_{t,s} \mathcal{A} \} \label{eqn:B-definition}
\end{equation}
where $\mathcal{A} \subset \R^{n-1}$ is a set wih fixed boundary of Lesbegue measure $0$ and $\rho_{t,s}>0$ is chosen such that

\begin{equation}
  \omega(\mathcal{B}_{t,s}(\mathcal{A},0)) = \frac{ \vol_{\R^{n-1}}\mathcal{A}}{(\#\mathcal{P}_{t,s}(g\overline{\vect{w}})^{\frac{n-1}{\delta_{\Gamma}}}}
\end{equation}
(for large $t$, $\rho_{t,s} \sim \vartheta^{-1/\delta_{\Gamma}}e^{-t}$). The random translations from the previous section will be replaced with random rotations on the sphere. Recall the map from \prettyref{thm:spherical-averages} for an open $\mathcal{U} \subset \R^{n-1}$, $\vect{x} \mapsto R(\vect{x})$ and let

\begin{equation}
  \mathcal{B}_{t,s}(\mathcal{A},\vect{x}) : = R(\vect{x})^{-1}(\mathcal{B}_{t,s}(\mathcal{A},0)).
\end{equation}
From which we define the random variable

\begin{equation}
  \mathcal{N}_{t,s}(\mathcal{A},\vect{x}, g\overline{\vect{w}}):= \#(\mathcal{P}_{t,s}(g\overline{\vect{w}}) \cap \mathcal{B}_{t,s}(\mathcal{A},\vect{x})).
\end{equation}
Finally, let 

\begin{equation}
  C_{\lambda,\mathcal{U}} := \int_{\mathcal{U}} \lambda'(\vect{x}) d\omega^{PS}_{\Gamma,g,\overline{K}}(\vect{x})
\end{equation}
With that we can describe the joint distribution for several test sets: $\mathcal{A}_1,...,\mathcal{A}_m$:


\begin{theorem} \label{thm:joint-distribution-sphere}
  Let $\mathcal{U} \subset \R^{n-1}$ be a nonempty open subset and let $R: \mathcal{U} \to K$ be a map as in \prettyref{thm:spherical-averages}. Let $\lambda$ be a compactly supported Borel probability measure on $\mathcal{U}$ with continuous density. Then for every $g \in G$, $s \in [0,\infty]$, $r=(r_1,...r_m) \in \Z^{m}_{> 0}$ and $\mathcal{A} = \mathcal{A}_1\times...\times \mathcal{A}_m$ with $\mathcal{A}_j \subset \R^{n-1}$ with boundary of Lebesgue measure $0$:

  \begin{equation}
    \lim_{t \to \infty}e^{(n-1-\delta_{\Gamma})t} \lambda (\{\vect{x} \in \mathcal{U} : \mathcal{N}_{t,s}(\mathcal{A}_j, \vect{x}; g\overline{\vect{w}}) = r_j \forall j\}) = E_{s}(r,\mathcal{A};g\overline{\vect{w}})
  \end{equation}
  where $E_s(r,\mathcal{A};g\overline{\vect{w}})$ is as in \prettyref{thm:joint-limit} with $A_{\lambda}$ replaced by $C_{\lambda,\mathcal{U}}$.

\end{theorem}

The proof of this theorem follows the same steps as \prettyref{thm:joint-limit} replacing the horospherical averages with the spherical ones proved in the previous section and \prettyref{lem:MV-Lemma-6} replaced with the following:


\begin{lemma}\label{lem:MV-Lemma-16}
  Under the hypotheses of \prettyref{thm:joint-distribution-sphere}, given $\epsilon>0$ there exists a $t_0<\infty$ and bounded subsets $\mathcal{A}_j^- \subset \mathcal{A}_j^+ \subset \R^{n-1}$ with boundary of measure $0$, such that:

  \begin{equation}
    \vol_{\R^{n-1}}(\mathcal{A}_j^+ \setminus \mathcal{A}_j^-) < \epsilon
  \end{equation}
  and for all $t\ge t_0$:

  \begin{equation}
    \#(a_tR(\vect{x})a_t\overline{\vect{w}} \cap \mathcal{Z}(\epsilon,s^{-},\mathcal{A}_j^-)) \le \mathcal{N}_{t,s}(\mathcal{A}_j,\vect{x};g\overline{\vect{w}})  \le  \#(a_tR(\vect{x})a_t\overline{\vect{w}} \cap \mathcal{Z}(-\epsilon,s+\epsilon,\mathcal{A}_j^+))
  \end{equation}
  with

  \begin{equation}
    s^- =
    \begin{cases}
      s-\epsilon & (s<\infty)\\
      \epsilon^{-1} & (s= \infty).
    \end{cases}
  \end{equation}

\end{lemma}

The proof of this Lemma is identical to that of {\cite[Lemma 16]{MarklofVinogradov2018}}. The one exception is the scaling in the definition of $\rho_{t,s}$ in (\ref{eqn:B-definition}). We therefore omit it.

\begin{proof}\phantom{\qedhere}[Proof of \prettyref{thm:main-theorem}]
  The proof is essentially an application of \prettyref{thm:joint-distribution-sphere}. Choose $m=1$ and $\mathcal{A} \subset \R^{n-1}$ to be a Euclidean ball of volume $\sigma$. Then set
  \begin{equation}
    \mathcal{B}_{t,s}(\mathcal{A},0) := \{ E(\vect{x})^{-1}(e^{-1} \vect{i})  : \vect{x} \in \rho_{t,s} \mathcal{A} \} = \mathcal{D}_{t,s}(\sigma, e^{-1}\vect{i}, g \overline{\vect{w}})
  \end{equation}
 Define the coordinate chart
 \begin{align}
   \begin{aligned}
   &\mathcal{U} \to S_1^{n-1}\\
   &\vect{x} \mapsto \vect{v} = R(\vect{x})^{-1}(e^{-1} \vect{i})
   \end{aligned}
 \end{align}
 for appropriate $\mathcal{U}$ and $R(\vect{x})$. Consider

 \begin{align}
   \begin{aligned}
   E_s(r,\sigma; g\overline{\vect{w}}) &= \lim_{t \to \infty} e^{(n-1-\delta_{\Gamma}t} \lambda (\{ \vect{v} \in S_1^{n-1} : \mathcal{N}_{t,s}(\sigma, \vect{v}; g\overline{\vect{w}}) = r \})\\
   &= \lim_{t \to \infty} e^{(n-1-\delta_{\Gamma}t} \lambda (\{ k \in K : \mathcal{N}_{t,s}(\sigma, k e^{-1}\vect{i}; g\overline{\vect{w}}) = r \})
   \end{aligned}
 \end{align}
 Applying the parameterisation $R:\mathcal{U} \to \overline{K}$ (and thus restricting the measure $\lambda$ so that the new density is $\lambda^\prime \chi_{R(\mathcal{U})}$) and using Lemma \ref{lem:polar}

 \begin{align}
   E_{s,\mathcal{U}}(r,\sigma; g\overline{\vect{w}}) = \lim_{t \to \infty} e^{(n-1-\delta_{\Gamma}t} \int_{\R^{n-1}}\chi_{\mathcal{U}}(\vect{x}) \lambda^\prime(R(\vect{x})) \chi_{\mathcal{A}}(R(\vect{x})) \left| \frac{\partial \tilde{\vect{x}}}{\partial \vect{x}} \right| |\vect{a}(\vect{x})|^{-(n-1)} d\vect{x}
 \end{align}
 Now applying \prettyref{thm:joint-distribution-sphere} with $\tilde{\lambda}'(\vect{x}) = \chi_{\mathcal{U}}(\vect{x})\lambda^\prime(R(\vect{x})) \left| \frac{\partial \tilde{\vect{x}}}{\partial \vect{x}} \right| |\vect{a}(\vect{x})|^{-(n-1)}$ implies

 \begin{equation}
   E_{s,\mathcal{U}}(r,\sigma;g\overline{\vect{w}}) = C_{\tilde{\lambda}, \mathcal{U}}\BR(\{\alpha \in G/\Gamma : \#(\alpha^{-1}\overline{\vect{w}} \cap \mathcal{Z}_0(s,\sigma))=r \})
 \end{equation}
 With

 \begin{align}
   C_{\tilde{\lambda},\mathcal{U}} &= \int_{\mathcal{U}} \lambda^\prime(R(\vect{x})) \left| \frac{\partial \tilde{\vect{x}}}{\partial \vect{x}} \right| |\vect{a}(\vect{x})|^{-(n-1)} d\omega^{PS}_{\Gamma,g,\overline{K}}(\vect{x})\\
    &= \int_{\overline{K}} \chi_{R(\mathcal{U})}(k) \lambda^\prime(k) d\mu^{PS}_{\Gamma g \overline{K}}(k)
 \end{align}
 By choosing suitable $\mathcal{U}$, partitioning $S^{n-1}_1$ we have thus proved Theorem \ref{thm:main-theorem}.  The continuity in $s$ and $\sigma$ and (\ref{eqn:sigma-limit}) follow from (\ref{eqn:continuous-in-A}).

\end{proof}


\subsection{Moment Generating Function}

Much like in section 4 the convergence result \prettyref{thm:spherical-averages} gives rise to a convergence result for the moment generating function for a non-cuspidal observer:

\begin{equation}
  \mathbb{G}_{t,s}(\tau_1,...,\tau_m;\mathcal{A}) := \int_{S^{n-1}} \mathbbm{1}(\mathcal{N}_{t,s}(\mathcal{A}_j,\vect{v}; g\overline{\vect{w}})\neq 0 ; \forall j)\exp \left( \sum_{j=1}^m \tau_j \mathcal{N}_{t,s}(\mathcal{A}_j,\vect{v}; g\overline{\vect{w}})\right) d\lambda(\vect{v}).
\end{equation} 

\begin{theorem} \label{thm:moment-generating-non-cuspidal}
  Let $\lambda$ be a probability measure on $S_1^{n-1}$ with continuous density. Then there exists a $c_0>0$ such that for all $\operatorname{Re}_{+}(\tau_1) + ... + \operatorname{Re}_{+}(\tau_m) < c_0$ and $s \in (0,\infty]$:

  \begin{equation}
    \lim_{t\to \infty} e^{(n-1-\delta_{\Gamma})t} \mathbb{G}_{t,s}(\tau_1,...,\tau_m; \mathcal{A}) = \frac{C_{\lambda}}{|\mathrm{m}^{BMS}|} \mathbb{G}_s(\tau_1,...,\tau_m;\mathcal{A}).
  \end{equation}

\end{theorem}

The proof of Theorem \ref{thm:moment-generating-non-cuspidal} is very similar to the proof \prettyref{thm:moment-generating}. The only difference is that \prettyref{lem:MV-Lemma-10} and \prettyref{lem:MV-Lemma-12} are replaced with Lemma \ref{lem:MV-Lemma-19} and Lemma \ref{lem:MV-Lemma-20} respectively. Recall the definition of the direction function $\varphi_{\vect{i}}(\vect{z})$ from the top of Section 2.3. For $B \subset S_1^{n-1}$ and $-\infty \le a < b < \infty$ define the cone

\begin{align}
  \mathcal{C}(a,b,B) : = \{ z \in \half^n \setminus \{\vect{i}\}, \varphi_{\vect{i}}(\vect{z}) : a < d(\vect{i},\vect{z}) \le b \}.
\end{align}


\begin{lemma} \label{lem:MV-Lemma-19}
  Fix $a \in \R$ and a bounded $\mathcal{A} \subset \R^{n-1}$. Then there exist positive constants $\zeta, \eta, t_0$ such that for all $g\in G$, $r\in \N_{>0}$, $t \ge t_0$

\begin{equation}
      \left[ \#(g\overline{\vect{w}} \cap \mathcal{C}(0,t, \mathcal{B}_{t,\infty}(\mathcal{A},0)))\ge r \right ] \Rightarrow \left[ \#(g\overline{\vect{w}} \cap \mathcal{C}(0,t-\zeta r+\eta, \mathcal{B}_{t,\infty}(\mathcal{A},0)))\ge 1 \right]
\end{equation}

\end{lemma}
As with \prettyref{lem:MV-Lemma-10}, this theorem is stated identically to {\cite[Lemma 19]{MarklofVinogradov2018}}, as the statement concerns only the definition of the spherical cone $\mathcal{C}$ and this is the same in both papers we omit the details.


\begin{lemma} \label{lem:MV-Lemma-20}
  Fix a bounded set $\mathcal{A} \subset \R^{n-1}$ and $\zeta$ and $\eta$ as in \prettyref{lem:MV-Lemma-19}. Let $\lambda$ be a Borel probabiility measure on $\mathcal{U}$ as in \prettyref{thm:joint-distribution-sphere}. Then there exists a $C$ such that for all $r\ge 0$

  \begin{equation}
    \sup_{t >0}e^{(n-1-\delta_{\Gamma})t} \int_{\mathcal{U}}\#(a_tR(\vect{x})g\overline{\vect{w}} \cap \mathcal{C}(0,t-\zeta r +\eta,\mathcal{B}_{t,\infty}(\mathcal{A},0)))d\lambda(\vect{x}) \le C e^{-\delta_{\Gamma}\zeta r}
    \end{equation}
\end{lemma}

\begin{proof}\phantom{\qedhere}
  The proof of this lemma is identical to that of {\cite[Lemma 20]{MarklofVinogradov2018}} with the one exception that we use (\ref{eqn:P-asymptotics}) rather than the analogous asymptotics.

  Replace $\mathcal{B}_{t,\infty}(\mathcal{A},0)$ with the ball $\mathcal{D}_t \subset S_1^{n-1}$ contianing it of volume $\omega(\mathcal{D}_t) = \sigma_0e^{-(n-1)t}$ for all $t\ge 0$ and some $\sigma_0$. We can bound this by

\begin{multline}
  \int_{\mathcal{U}}\#(a_t R(\vect{x})g\overline{\vect{w}} \cap \mathcal{C}(0,t-\zeta r +\eta,\mathcal{B}_{t,\infty}(\mathcal{A},0)))d\lambda(\vect{x}) \\ \le C_2\int_K \#(a_tkg\overline{\vect{w}} \cap \mathcal{C}(0,t-\zeta r +\eta, \mathcal{D}_t))d\mu^{Haar}_K(k).
\end{multline}
Using the definition of $\mathcal{C}(\cdot,\cdot,\cdot)$,

\begin{align*}
  &C_2\int_K \#(a_tkg\overline{\vect{w}} \cap \mathcal{C}(0,t-\zeta r +\eta, \mathcal{D}_t))d\mu^{Haar}_K(k)\\
  &\phantom{++++++++++}\le \sigma_0 e^{-(n-1)t}\#\{\gamma  \in  \Gamma/\Gamma_{\vect{w}}, \; :  d(g \gamma\vect{w} ) \le e^{t-\zeta r +\eta}\}.
\end{align*}
By (\ref{eqn:P-asymptotics}) we conclude that

\begin{equation}
  \int_{\mathcal{U}}\#(a_t R(\vect{x})g\overline{\vect{w}} \cap \mathcal{C}(0,t-\zeta r +\eta,\mathcal{B}_{t,\infty}(\mathcal{A},0)))d\lambda(\vect{x}) \le C\sigma_0 e^{-(n-1)t} \max(1,e^{\delta_{\Gamma}(t-\zeta r)}).
\end{equation}
Lemma \ref{lem:MV-Lemma-20} follows from here.

\end{proof}

  \section{Applications to Moments, Two Point Correlation Function and Gap Statistics}

\subsection{Convergence of Moments}

Once again analogous to \cite{MarklofVinogradov2018}, we note that Theorem \ref{thm:moment-generating} and Theorem \ref{thm:moment-generating-non-cuspidal} each gives rise to a corollary concerning the convergence of moments (we state them here as one):

For an observer on the boundary of hyperbolic space consider the mixed-moment:

\begin{equation}
  \mathbb{M}_{t,s}^{\infty} (\beta_1,...,\beta_m;\mathcal{A}):= \int_{\T^{n-1}} \prod_{j=1}^m (\mathcal{N}_{t,s}^{\infty}(\mathcal{A}_j,\vect{x};\overline{\vect{w}}))^{\beta_j}d\lambda(\vect{x})
\end{equation}
for all $\beta_j \in \R_{\ge 0}$ with limit moment:

\begin{equation} \label{limit moment}
  \mathbb{M}_{s} (\beta_1,...,\beta_m;\mathcal{A}):=\sum_{r_1,...,r_m=1}^\infty r_1^{\beta_1}... r_m^{\beta_{m}}E_{s}(r,\mathcal{A}; \overline{\vect{w}}).
\end{equation}
For a non-cuspidal observer we define:

\begin{equation}
  \mathbb{M}_{t,s} (\beta_1,...,\beta_m;\mathcal{A}):= \int_{S_1^{n-1}} \prod_{j=1}^m (\mathcal{N}_{t,s}(\mathcal{A}_j,\vect{x};\overline{\vect{w}}))^{\beta_j}d\lambda(\vect{x})
\end{equation}
for all $\beta_j \in \R_{\ge 0}$ (the limit moment is the same). Hence the following corollary follows from Theorem \ref{thm:moment-generating} and Theorem \ref{thm:moment-generating-non-cuspidal}


\begin{corollary} \label{cor:convergence-of-moments}
  Let $\lambda$ be a probability measure on $\T^l \times \R^{n-1-l}$, with a bounded continuous density with respect to Lebesgue, and $\mathcal{A} = \mathcal{A}_1\times ... \times \mathcal{A}_m$ with $\mathcal{A}_j \subset \R^{n-1}$ bounded with boundary of Lebesgue measure zero. Then for all $\beta_1,...,\beta_m\in \R_{\ge 0}$, $s \in[0,\infty]$:
  \begin{equation}
    \mathbb{M}_s(\beta_1,...,\beta_m; \mathcal{A}) < \infty
  \end{equation}
  \begin{equation}
    \lim_{t\to \infty} e^{(n-1-\delta_{\Gamma})}\mathbb{M}_{t,s}^{\infty}(\beta_1,...,\beta_m;\mathcal{A}) = \frac{A_{\lambda}}{|\mathrm{m}^{BMS}|}\mathbb{M}_s(\beta_1,...,\beta_m;\mathcal{A}). \label{eqn:limit-moments}
  \end{equation}
  For an observer in $\half^n$ with $\lambda$ a probability measure on $S_1^{n-1}$ with bounded continuous density with respect to Lebesgue, the same conclusion holds with (\ref{eqn:limit-moments}) replaced with

  \begin{equation}
    \lim_{t\to \infty} e^{(n-1-\delta_{\Gamma})}\mathbb{M}_{t,s}(\beta_1,...,\beta_m;\mathcal{A}) = \frac{C_{\lambda}}{|\mathrm{m}^{BMS}|}\mathbb{M}_s(\beta_1,...,\beta_m;\mathcal{A}).
  \end{equation}

\end{corollary}


\subsection{Two-Point Correlation Function}

 We will work in the case of an observer on the boundary (thus w.l.o.g at $\vect{\infty}$), note that this then applies to the sphere packing case. The case of an observer in the interior can be treated similarly however working on $S_1^{n-1}$ rather than $\T^{n-1}$ makes the problem more complex. Furthermore we will work in the special case of $\T^{n-1}$ rather than $\T^{l}\times \R^{n-1-l}$, however that case follows similarly. As we will use it throughout recall that $\mathcal{B}_r(\vect{x}) \subset \T^{n-1}$ denotes the ball of size $r$ around $\vect{x}$.

Consider the points in $\mathcal{P}^{\vect{\infty}}_t(\overline{\vect{w}})$ and label them $\{\vect{x}_i\}_{i=1}^{N_t} \subset \T^{n-1}$ where $N_t = \#\mathcal{P}^{\vect{\infty}}_t(\overline{\vect{w}}) \sim c_0^{-1}e^{\delta_{\Gamma}t}$ (in the notation of \prettyref{thm:OhShahAsymptotics} $c_0^{-1} =  \vartheta |\mu_{\Gamma g \overline{H}}^{PS}|$). We consider first the two-point correlation function, for $f \in \mathcal{C}_0(\T^{n-1})$,

\begin{equation}\label{2 pt cont}
  R_2(f)(t) := \frac{c_0}{e^{\delta_{\Gamma}t}} \sum^{N_t}_{\substack{i,j = 1, \\i \neq j}}f(e^t (\vect{x}_i - \vect{x}_j))
\end{equation}
As was done in {\cite[Appendix A]{El-BazMarklofVinogradov2015}} (their analysis is for more general functions but we will restrict to this simpler case), we can approximate $f$ from above and below by a finite linear combination of functions of the form

\begin{equation}
  \tilde{f}(\vect{z}) = \sum_{k=1}^p \gamma_k \int_{\vect{x}\in \T^{n-1}} \left(\chi_{\mathcal{R}_{1,k}}( \vect{z} +\vect{x})\chi_{\mathcal{R}_{2,k}}(\vect{x})\right)d\vect{x}
\end{equation}
where $\mathcal{R}_{i,k}$ are rectangular boxes. In other words, for any $\epsilon$, there exists a $p<\infty$, $\{\mathcal{R}_{i,k}\}_{k=1}^p$ bounded, and $\{\gamma_k^u\}_{k=1}^{p}, \{\gamma_k^l\}_{k=1}^p$ such that

\begin{align}\label{R approx}
  \begin{aligned}
    &\sum_{k=1}^p \gamma_k^l \int_{\vect{x}\in \T^{n-1}}\left( \chi_{\mathcal{R}_{1,k}}( \vect{z} +\vect{x})\chi_{\mathcal{R}_{2,k}}(\vect{x}\right))d\vect{x}
    \le f(\vect{z})  \\
    &\phantom{++++++++++++++++}\le \sum_{k=1}^p \gamma_k^u \int_{\vect{x}\in \T^{n-1}}\left( \chi_{\mathcal{R}_{1,k}}( \vect{z} +\vect{x})\chi_{\mathcal{R}_{2,k}}(\vect{x})\right)d\vect{x},
  \end{aligned}
\end{align}
and 

\begin{equation} \label{R approx eps}
  \sum_{k=1}^p (\gamma_k^u-\gamma_k^l) \int_{\vect{x}\in \T^{n-1}} \left(\chi_{\mathcal{R}_{1,k}}( \vect{z} +\vect{x})\chi_{\mathcal{R}_{2,k}}(\vect{x})\right)d\vect{x} \le \epsilon.
\end{equation}

Hence we can approximate $R_2(f)(t)$ by functions of the form

\begin{align} \label{approximating chars}
  \begin{aligned}
   &c_0e^{(-\delta_{\Gamma})t} \sum_{k=1}^p  \gamma_k \int_{\vect{x}\in \T^{n-1}}\left( \sum^{N_t}_{\substack{i,j = 1, \\i \neq j}} \chi_{\mathcal{R}_{1,k}}( e^t(\vect{x}_i-\vect{x}_j) +\vect{x}) \chi_{\mathcal{R}_{2,k}}(\vect{x}) \right)d\vect{x}\\
   &\phantom{++}= c_0e^{(n-1-\delta_{\Gamma})t} \sum_{k=1}^p  \gamma_k \int_{\vect{x}\in \T^{n-1}}\left( \sum^{N_t}_{\substack{i,j = 1, \\i \neq j}} \chi_{e^{-t}\mathcal{R}_{1,k}}( \vect{x}_i +\vect{x}) \chi_{e^{-t}\mathcal{R}_{2,k}}(\vect{x}_j+ \vect{x})\right) d\vect{x}\\
    &\phantom{++}=  c_0e^{(n-1-\delta_{\Gamma})t} \sum_{k=1}^p  \gamma_k\cdot \\
    &\phantom{+++}\int_{\vect{x}\in \T^{n-1}} \left( \mathcal{N}_{t,\infty}^{\infty}(\mathcal{R}_{1,k},\vect{x}; \overline{\vect{w}})\mathcal{N}_{t,\infty}^{\infty}(\mathcal{R}_{2,k},\vect{x}; \overline{\vect{w}}) - \mathcal{N}_{t,\infty}^{\infty}(\mathcal{R}_{1,k}\cap\mathcal{R}_{2,k} ,\vect{x}; \overline{\vect{w}})\right) d\vect{x}.
    \end{aligned}
\end{align}
Using Corollary \prettyref{cor:convergence-of-moments} we know 

\begin{align}
  \begin{aligned}
    &\lim_{t\to \infty} c_0e^{(n-1-\delta_{\Gamma})t} \sum_{k=1}^p  \gamma_k \cdot\\
    &\phantom{++}\left(\int_{\vect{x}\in \T^{n-1}} \mathcal{N}_{t,\infty}^{\infty}(\mathcal{R}_{1,k},\vect{x}; \overline{\vect{w}})\mathcal{N}_{t,\infty}^{\infty}(\mathcal{R}_{2,k},\vect{x}; \overline{\vect{w}}) - \mathcal{N}_{t,\infty}^{\infty}(\mathcal{R}_{1,k}\cap\mathcal{R}_{2,k} ,\vect{x}; \overline{\vect{w}}) d\vect{x}\right)\\ 
   &\phantom{+++++++}= \frac{A_{\lambda}c_0}{|\BMS|} \sum_{k=1}^p \left(\gamma_k (\mathbb{M}_{\infty}(1,1; \mathcal{R}_{1,k}\times \mathcal{R}_{2,k}) - \mathbb{M}_{\infty}(1,\mathcal{R}_{1,k}\cap\mathcal{R}_{2,k}))\right). \label{limit approx}
   \end{aligned}
\end{align}
Moreover, since $\mathbb{M}_{\infty}(\beta_1, \dots, \beta_n, \mathcal{A})$ is finite provided the sets $\mathcal{A}$ have finite area, for any $\varrho>0$ there exist $p, \{\mathcal{R}_k\}_{k=1}^p,\{\gamma_k^u\}^p_{k=1},\{\gamma_{k}^u\}_{k=1}^p$ such that 

\begin{equation}
  \sum_{k=1}^p \left((\gamma_k^u-\gamma_k^l) (\mathbb{M}_{\infty}(1,1; \mathcal{R}_{1,k}\times \mathcal{R}_{2,k}) - \mathbb{M}_{\infty}(1,\mathcal{R}_{1,k}\cap\mathcal{R}_{2,k}))\right) \le \varrho
\end{equation}
Hence the approximations from above and below converge in the limit $t\to \infty$ as well. Hence the limit $\lim_{t\to \infty}R_2(f)(t)$ exists. By an approximation argument if $f$ is an indicator function $\lim_{t\to \infty} R_2(f)(t)$ also exists. Thus
\begin{equation}\label{2 pt}
  R_2(\xi) := \lim_{t \to \infty} \frac{c_0}{e^{\delta_{\Gamma}t}} \sum^{N_t}_{i,j = 1 , \;i \neq j} \mathbbm{1}\left(\vect{x}_j \in \mathcal{B}_{\xi e^{-t}}(\vect{x}_i) \right)
\end{equation}
has a limit for every fixed $\xi$. It follows, as noted in the appendix of \cite{MarklofVinogradov2018} that

\begin{equation}
  R_2(\xi) = \lim_{\epsilon\to 0} \frac{c}{\epsilon^{n-1}}\left[ \mathbb{M}_{\infty}(1,1;\mathcal{B}_{\xi}(0) \times \mathcal{B}_{\epsilon}(0)) - \mathbb{M}_{\infty}(1;\mathcal{B}_{\epsilon}(0))\right] \label{eqn:epsilon-limit-M}
\end{equation}
where we have again used Corollary \prettyref{cor:convergence-of-moments} and set $c=c_0\frac{A_{\lambda}}{|\mathrm{m}^{BMS}|}$.

Moreover we can write
\begin{equation} 
\mathbb{M}_{\infty}(1,1;\mathcal{B}_{\xi}(0) \times \mathcal{B}_{\epsilon}(0)) - \mathbb{M}_{\infty}(1;\mathcal{B}_{\epsilon}(0)) =  \frac{1}{\#\Gamma_{\vect{w}}}\sum_{\substack{\gamma \in  \Gamma / \Gamma_{\vect{w}} \\ \gamma \neq \Gamma_{\vect{w}}}} F_{\gamma,\epsilon}(\vartheta^{-1}\xi)
\end{equation}
where

\begin{equation}
  F_{\gamma,\epsilon}(\vartheta^{-1} \xi) : = \int_{G} \mathbbm{1}(\alpha^{-1} \gamma \vect{w} \in \mathcal{Z}(\infty, \mathcal{B}_{\xi}(0)))\mathbbm{1}(\alpha^{-1}\vect{w} \in \mathcal{Z}(\infty,\mathcal{B}_{\epsilon}(0)))d\mathrm{m}^{BR}(\alpha),
\end{equation}
here $\mathcal{B}_{r}(0)$ is the ball of radius $r$ around $0$ in $\partial \half^n$. 

Applying the same Iwasawa decomposition and change of coordinates as was done in the proof of Proposition \prettyref{prop:MV-Lemma-5} gives

\begin{multline}\label{F gamma}
  F_{\gamma,\epsilon}(\vartheta^{-1} \xi) = \int_{KAN_+} \mathbbm{1}( g_{\vect{w}}\alpha^{-1} g_{\vect{w}}^{-1}\gamma \vect{w} \in \mathcal{Z}(\infty, \mathcal{B}_{\xi}(0)))\cdot\\ \mathbbm{1}( g_{\vect{w}}n_+a_{-r}k\vect{i} \in \mathcal{Z}(\infty,\mathcal{B}_{\epsilon}(0)))e^{-\delta_{\Gamma}r}d\mu^{Haar}_{N_+}dr d\nu_{\vect{i}}^{\vect{w}}(k \vect{X}_{\vect{i}}^-),
\end{multline}
recall $\nu^{\vect{w}}$ is the conformal density associated to the subgroup $\Gamma^{\vect{w}}$ (see the proof of Proposition \prettyref{prop:MV-Lemma-5}). Note that $g_{\vect{w}} \in G/K \cong AN_+$ which we write as $a_{r_{\vect{w}}}n_+(\vect{x})$. Hence

\begin{equation}
  g_{\vect{w}}n_+(\vect{x}) a_{-r} k \vect{i} = g_{\vect{w}} a_{-r} \vect{i} + e^{-r_{\vect{w}}}\vect{x}.
\end{equation}
Hence

\begin{multline}\label{F gamma 2}
  F_{\gamma,\epsilon}(\vartheta^{-1}\xi) = \int_{KA\R^{n-1}} \mathbbm{1}(g_{\vect{w}}a_{-r}g_{\vect{w}}^{-1}\gamma\vect{w} \in \mathcal{Z}(\infty, \mathcal{B}_{\xi}(0))-\vect{x}e^{-r_{\vect{w}}})\cdot\\ \mathbbm{1}( g_{\vect{w}}a_{-r}\vect{i} \in \mathcal{Z}(\infty,\mathcal{B}_{\epsilon}(0))- \vect{x}e^{-r_{\vect{w}}})e^{-\delta_{\Gamma}r}d\vect{x}dr d\nu_{\vect{i}}^{\vect{w}}(k \vect{X}_{\vect{i}}^-),
\end{multline}
Thus taking the limit

\begin{multline}
  \lim_{\epsilon \to 0}\frac{c}{\epsilon^{n-1}}F_{\gamma,\epsilon}(\vartheta^{-1}\xi) = c\int_{KA} \mathbbm{1}( a_{r_{\vect{w}}-r}kg_{\vect{w}}^{-1}\gamma\vect{w} \in \mathcal{Z}(\infty, \mathcal{B}_{\xi}(0)))\cdot\\ \mathbbm{1}(r_{\vect{w}}-r>0)e^{(n-1)r_{\vect{w}}-\delta_{\Gamma}r}dr d\nu_{\vect{i}}^{\vect{w}}(k \vect{X}_{\vect{i}}^-),
\end{multline}
Simplifying then gives

\begin{align} \label{epsilon trick}
  \begin{aligned}
    &\lim_{\epsilon \to 0}\frac{c}{\epsilon^{n-1}}F_{\gamma,\epsilon}(\vartheta^{-1}\xi) \\
    &\phantom{++++++}= c\int_{K\R_{>0}} \mathbbm{1}(  a_{-r}kg_{\vect{w}}^{-1} \gamma\vect{w} \in \mathcal{Z}(\infty, \mathcal{B}_{\xi}(0)))e^{(n-1-\delta_{\Gamma})r_{\vect{w}}-\delta_{\Gamma}r}dr d\nu_{\vect{i}}^{\vect{w}}(k \vect{X}_{\vect{i}}^-).
    \end{aligned}
\end{align}
Hence 

\begin{align} \label{limit in t}
  \begin{aligned}
    &R_2(\xi)\\
    &\phantom{++}= \frac{2c}{\#\Gamma_{\vect{w}}} \sum_{\substack{\gamma \in  \Gamma/ \Gamma_{\vect{w}}\\ \gamma \neq \Gamma_{\vect{w}}}} \int_{K\R_{>0}} \mathbbm{1}(a_{-r} k g_{\vect{w}}^{-1}\gamma\vect{w}  \in \mathcal{Z}(\infty, \mathcal{B}_{\xi}(0)))e^{(n-1-\delta_{\Gamma})r_{\vect{w}}-\delta_{\Gamma}r}dr d\nu_{\vect{i}}^{\vect{w}}(k \vect{X}_{\vect{i}}^-).
    \end{aligned}
\end{align}

Now, to evaluate whether $R_2$ is continuous in $\xi$, take $\xi>\xi'$ and consider the difference

\begin{multline} \label{R2 diff}
  \left| R_2(\xi)-R_2(\xi')\right| = \frac{2c}{\#\Gamma_{\vect{w}}} \sum_{\substack{\gamma \in \Gamma / \Gamma_{\vect{w}}  \\ \gamma \neq \Gamma_{\vect{w}}}} \int_{K\R_{>0}} \mathbbm{1}(a_{-r} k g_{\vect{w}}^{-1}\gamma\vect{w}  \in \mathcal{Z}(\infty, \mathcal{B}_{\xi}(0)\setminus \mathcal{B}_{\xi'}(0)))\cdot \\e^{(n-1-\delta_{\Gamma})r_{\vect{w}}-\delta_{\Gamma}r}dr d\nu_{\vect{i}}^{\vect{w}}(k \vect{X}_{\vect{i}}^-).
\end{multline}
Suppose we are working in dimension $n=2$. In that case $\mathcal{Z}(\infty, \mathcal{B}_{\xi}(0)\setminus \mathcal{B}_{\xi'}(0))$ converges to 2 vertical line segments. Hence in the limit as $\xi' \to \xi$ for fixed $r$ there are at most $4$ rotations such that the point hits these four line segments. However since the measure $\nu_{\vect{i}}$ is non-atomic (see \cite{Sullivan1984}) it must give $0$ mass to these $4$ rotations. Hence the difference in the left hand side of \eqref{R2 diff} converges to $0$ and the two-point correlation function is continuous.

A similar argument implies, in general dimension $n > 2$, if $\delta_{\Gamma}> n-2$ then the difference in \eqref{R2 diff} also goes to $0$ and the two-point correlation function is continuous. The argument is essentially the same: the projection of the set $\mathcal{Z}(\infty,\mathcal{B}_{\xi}(0)\setminus \mathcal{B}_{\xi'}(0))$ to the boundary will be an $(n-2)$-sphere. Hence since the dimension of the limit set is larger than $n-2$ and the conformal density $\nu_{\vect{i}}$ is supported on the limit set (and finite), the above difference must go to $0$.

However, if $\delta_{\Gamma} \le n-2$ the continuity of $R_2$ will depend on the geometry of the limit set.


\subsection{Nearest Neighbor Statistics}

We will now use a similar method as for the two-point correlation function to write down an explicit formula for the nearest neighbor statistics of the point set $\mathcal{P}^{\vect{\infty}}_t(\overline{\vect{w}})$. In section 8.4 we will use a trick which works only in 2 dimensions to say something more about the gap statistics (i.e about the nearest neighbor \emph{to the right} statistics) however here we continue to work in general dimension $n$. 

Define the limiting cumulative nearest neighbor distribution to be

\begin{equation}
  \mathcal{J}(L) := \lim_{t\to \infty}\mathcal{J}_t(L) := \lim_{t\to \infty}\frac{1}{N_t} \sum_{i=1}^{N_t} \mathbbm{1}(\#(\mathcal{B}_{Le^{-t}}(\vect{x}_i) \cap \mathcal{P}^{\vect{\infty}}_t(\overline{\vect{w}}))=1 ).
\end{equation}
To determine the limiting behavior we will perform a similar trick as was used for the two-point correlation function. Again, writing $N_t \sim c_0^{-1} e^{\delta_{\Gamma}t}$
\begin{align*}
  \begin{aligned}
    &\mathcal{J}_t(L)\\
    &= \lim_{\epsilon \to 0} \frac{c_0}{e^{t\delta_{\Gamma}}\epsilon^{n-1}} \int_{\vect{x} \in \T^{n-1}} \mathbbm{1}(\#(\mathcal{B}_{\epsilon}(\vect{x})\cap \mathcal{P}^{\vect{\infty}}_t(\overline{\vect{w}})) = 1) \mathbbm{1}(\#(\mathcal{B}_{Le^{-t}}(\vect{x})\cap \mathcal{P}^{\vect{\infty}}_t(\overline{\vect{w}})) = 1) d\vect{x}\\
    &= \lim_{\epsilon \to 0} \frac{c_0e^{(n-1-\delta_{\Gamma})t}}{\epsilon^{n-1}} \int_{\vect{x} \in \T^{n-1}} \mathbbm{1}(\#(\mathcal{B}_{\epsilon e^{-t}}(\vect{x})\cap \mathcal{P}^{\vect{\infty}}_t(\overline{\vect{w}})) = 1)\cdot \\
    &\phantom{++++++++++++++++++++++++}\mathbbm{1}(\#(\mathcal{B}_{Le^{-t}}(\vect{x})\cap \mathcal{P}^{\vect{\infty}}_t(\overline{\vect{w}}) =1 )d\vect{x}.
   \end{aligned}
\end{align*}
Using that our test set $\mathcal{B}_{e^{-t}L}(\vect{x})$ and $\mathcal{B}_{e^{-t}\epsilon}(\vect{x})$ have the same scaling as $\mathcal{B}_{t,s}$ \eqref{eqn:definition-B} together with the asymptotic $\#\mathcal{P}^{\vect{\infty}}_t(\overline{\vect{w}}) \sim c_0^{-1}e^{\delta_{\Gamma}t}$ we can apply \prettyref{thm:joint-limit} to take the limit $t\to \infty$ (and as above, using the linearity in $\epsilon$ to exchange the limits), giving

\begin{equation}
  \mathcal{J}(L) = \lim_{\epsilon \to 0} \frac{c_0}{\epsilon^{n-1}} E_{\infty}((1,1), \mathcal{B}_{\epsilon}(0) \times \mathcal{B}_{L}(0) ; \overline{\vect{w}}),
\end{equation}
which is then equal

\begin{align*}
  \mathcal{J}(L) 
  &= \lim_{\epsilon \to 0} \frac{\vartheta }{|\BMS|\epsilon^{n-1}}\\
  &\phantom{++}\BR \left(\{ \alpha \in \Gamma \backslash G : \#(\alpha^{-1}\overline{\vect{w}} \cap \mathcal{Z}(\infty, \mathcal{B}_{\epsilon}(0)))= 1, \; \#(\alpha^{-1}\overline{\vect{w}} \cap \mathcal{Z}(\infty, \mathcal{B}_{L}(0)))= 1 \}\right)  \\
  &= \lim_{\epsilon \to 0} \frac{\vartheta}{|\BMS|\epsilon^{n-1}} \cdot \\
  &\int_G \mathbbm{1}(\alpha^{-1}\vect{w} \in \mathcal{Z}(\infty,\mathcal{B}_{\epsilon}(0))) \prod_{\substack{\gamma \in \Gamma / \Gamma_{\vect{w}} \\ \gamma \neq \Gamma_{\vect{w}}}} \left( 1 - \mathbbm{1}(\alpha^{-1} \gamma\vect{w} \in \mathcal{Z}(\infty,\mathcal{B}_L(0)))\right) d\BR(\alpha) 
\end{align*}
Hence, using the same trick as was done for \eqref{F gamma} we can write this

\begin{align}
  \begin{aligned}
    &\mathcal{J}(L) = \frac{\vartheta}{|\BMS|}\\
    &\phantom{+++}\int_{K\R_{>0}} \prod_{\substack{\gamma \in  \Gamma / \Gamma_{\vect{w}} \\ \gamma \neq \Gamma_{\vect{w}}}} \left( 1 - \mathbbm{1}( a_{-r}k g_{\vect{w}}^{-1} \gamma \vect{w}  \in \mathcal{Z}(\infty,\mathcal{B}_L(0)))\right) e^{(n-1-\delta_{\Gamma})r_{\vect{w}}}e^{\delta_{\Gamma} r} dr d\nu_{\vect{i}}^{\vect{w}}(k\vect{X}_{\vect{i}}^-).
    \end{aligned}
\end{align}


\subsection{Gap Statistics}

In this last section we prove, for the discrete subgroups considered here, the same result as is found in \cite{Zhang2017} for Schottky groups. That is, we prove \prettyref{thm:gap-distributions-intro} from the introduction. Define the \emph{cumulative} gap distribution to be:

\begin{equation}
  F_t(L):= \frac{1}{N_t} \sum_{j=1}^{N_t}\mathbbm{1}\left(s_j \ge L\right).
\end{equation}
Applying the same argument as we used above for the nearest neighbor distribution, we can write the limiting cumulative distribution:

\begin{align}
  \begin{aligned}
  &F(L) := F_t(L)\\
       &\phantom{++}= \lim_{t\to \infty} \frac{1}{N_t}\sum_{i=1}^{N_t} \mathbbm{1}(\#([\vect{x}_i,\vect{x}_i +Le^t) \cap \mathcal{P}_t^{\infty}(\overline{\vect{w}}))=1)\\
       &\phantom{++}= \lim_{\epsilon \to 0 } \frac{c_0}{\epsilon} E_{\infty}\left((1,1),[0,\epsilon) \times [0,L); \overline{\vect{w}} \right)\\
       &\phantom{}= \frac{\vartheta}{|\BMS|} \int_{K\R_{>0}} \prod_{\substack{\gamma \in  \Gamma/\Gamma_{\vect{w}} \\ \gamma \neq \Gamma_{\vect{w}}}} \left( \mathbbm{1}( a_{-r}kg_{\vect{w}}^{-1}\gamma\vect{w}  \not\in \mathcal{Z}(\infty,[0,L)))\right) e^{(1-\delta_{\Gamma})r_{\vect{w}}}e^{\delta_{\Gamma} r} dr d\nu_{\vect{i}}^{\vect{w}}(k\vect{X}_{\vect{i}}^-). \label{F(L) fin}
  \end{aligned}
\end{align}

 A classical argument (explained in some detail in \cite{Marklof2007}) shows that the gap distribution is the derivative of the $E_s(r,\sigma, \overline{\vect{w}})$  for $r=0$. As we have not treated the case $r=0$ let

\begin{equation}
  E(L,\overline{\vect{w}}) := \sum_{r=1}^{\infty} E_s(r,L;\overline{\vect{w}})
\end{equation}
in which case the following lemma is a direct consequence of the argument in \cite{Marklof2007}.


\begin{lemma} \label{lem:cumulative}
  In the present setting, for any $L >0$

  \begin{equation}
  F(L)  = - \frac{d}{dL}E(L,\overline{\vect{w}}). \label{eqn:cumulative}
  \end{equation}
\end{lemma}

With that we prove \prettyref{thm:gap-distributions-intro} restated here for convenience.


\begin{namedtheorem}[\prettyref{thm:gap-distributions-intro}] \label{thm:gap-distribution}
  The limiting function $F(L)$ exists, is monotone decreasing and continuous (including at $0$).    

  Moreover we show that the gap distribution satisfies the following formula

\begin{equation}\label{gap explicit2}
  F(L) = C_{\vect{w}} \int_0^{\infty} e^{\delta_{\Gamma}r}\int_0^{\pi}\prod_{\substack{\gamma \in \Gamma/\Gamma_{\vect{w}} \\ \gamma \neq \Gamma_{\vect{w}}}} \left( 1- \chi_{\mathcal{E}(\gamma)}(r,\theta)\right) d\nu_{\vect{i}}(\theta) dr,
\end{equation}
where $C_{\vect{w}}$ is an explicit constant, $\mathcal{E}(\gamma)$ is an explicit set depending on the choice of $\Gamma$, and here and throughout $\chi_{\mathcal{A}}$ is the characteristic function of the set $\mathcal{A}$.


\end{namedtheorem}

\begin{proof}\phantom{\qedhere}

  The calculation above \prettyref{lem:cumulative} establishes the existence of $F$ and the fact that it is monotone decreasing follows from \prettyref{lem:cumulative}.  

  Moreover the argument for continuity follows from the comment at the end of section 8.2 for the two point correlation function. I.e for $L > L'$ we consider the difference

  \begin{align}
    \begin{aligned}
     & F(L') - F(L) \\
     & = \frac{\vartheta}{|\BMS|} \int_{K\R_{>0}} \prod_{\substack{\gamma \in  \Gamma/\Gamma_{\vect{w}} \\ \gamma \neq \Gamma_{\vect{w}}}} \left( \mathbbm{1}( a_{-r} k g_{\vect{w}}^{-1}\gamma\vect{w} \in \mathcal{Z}(\infty,[L',L)))\right) e^{(1-\delta_{\Gamma})r_{\vect{w}}}e^{\delta_{\Gamma} r} dr d\nu_{\vect{i}}^{\vect{w}}(k\vect{X}_{\vect{i}}^-).
    \end{aligned}
  \end{align}
  Again, as $L \to L'$ the indicator function inside the integral becomes the indicator function that the point $a_{-r} k g_{\vect{w}}^{-1}\gamma\vect{w}$ lies on a line segment. Since the line segment is transversal to the rotation, for $a_r$ fixed this can only happen for (at most) 2 rotations. Since $\nu_{\vect{i}}$ is non-atomic this event has measure $0$.

  The proof of \eqref{gap explicit2} is the content of the next subsection.

\end{proof}

\subsection{Explicit Calculations for the Gap Distribution}

In \eqref{F(L) fin} we used the Iwasawa decomposition and $\vect{w} = g_{\vect{w}}\vect{i}$. In fact we had a choice of $g_{\vect{w}}$. Thus in the equation

\begin{equation} \label{gap explicit 1}
  F(L) = \frac{\vartheta}{|\BMS|}\int_{K\R_{>0}} \prod_{\substack{\gamma \in \Gamma/ \Gamma_{\vect{w}} \\ \gamma \neq \Gamma_{\vect{w}}}}  \mathbbm{1}(a_{-r}kg_{\vect{w}}^{-1}\gamma\vect{i}  \not\in \mathcal{Z}(\infty,[0,L))) e^{\delta_{\Gamma} r} e^{(1-\delta_{\Gamma})r_{\vect{w}}} dr d\nu_{\vect{i}}^{\vect{w}}(k\vect{X}_{\vect{i}}^-).
\end{equation}
choose $g_{\vect{w}}^{-1}$ such that, in polar coordinates $g_{\vect{w}}^{-1}\gamma\vect{i} = \kappa(\gamma)(e^{l(\gamma)}\vect{i})$ where $l(\gamma) = d(\vect{w}, \gamma\vect{w})$ and $\kappa(\gamma)$ is a rotation. In which case \eqref{gap explicit 1} becomes

\begin{align} \label{gap explicit 2}
  \begin{aligned}
    &F(L) = \frac{\vartheta}{|\BMS|}\\
    &\phantom{+++}\int_{K\R_{>0}} \prod_{\substack{\gamma \in  \Gamma/ \Gamma_{\vect{w}} \\ \gamma \neq \Gamma_{\vect{w}}}}  \mathbbm{1}(   a_{-r}k\kappa(\gamma)(e^{l(\gamma)}\vect{i}) \not\in \mathcal{Z}(\infty,[0,L)))  e^{\delta_{\Gamma} r} e^{(1-\delta_{\Gamma})r_{\vect{w}}} dr d\nu_{\vect{i}}^{\vect{w}}(k\vect{X}_{\vect{i}}^-).
  \end{aligned}
\end{align}
Unfortunately we cannot remove the factor $\kappa(\gamma)$, while the conformal density is invariant under the action of $\Gamma$ the terms in the product inside the integral are not independent. However, given the group element, $\kappa(\gamma)$ and $l(\gamma)$ are explicit. We can now use a change of variables as in the appendix of \cite{MarklofVinogradov2018} with 

\begin{equation}
  k = k(\theta) = \mat{\cos\theta}{-\sin\theta}{\sin\theta}{\cos\theta}.
\end{equation}
With that, and writing $\kappa(\gamma) = k(\theta(\gamma))$, the constraint

\begin{equation}
  \mathcal{D}(\gamma) := \{(r,\theta) : a_{-r} k(\theta + \theta(\gamma))(e^{l(\gamma)}\vect{i}) \in \mathcal{Z}(\infty,[0,L))\}
\end{equation}
is equal
\begin{multline}\label{k r constraint}
  \mathcal{E}(\gamma) = \left\{ (r,\theta) : \frac{e^{-r}}{\sinh l(\gamma) \cos 2(\theta+\theta(\gamma))-\cosh l(\gamma) } > 1,\right. \\ \left. 0 \le \frac{e^{-r} \sinh l(\gamma) \sin 2(\theta+\theta(\gamma))}{\cosh l(\gamma) - \sinh l(\gamma)\cos 2(\theta+\theta(\gamma))} <\vartheta^{-1/ \delta_{\Gamma}}L \right\}.
\end{multline}
In which case we have the following theorem


\begin{theorem} \label{thm:gap dist explicit} 
    For $L >0$, the cumulative gap distribution can be written
    \begin{equation}\label{gap dist explicit}
      F(L) = \frac{\vartheta e^{(1-\delta_{\Gamma})r_{\vect{w}}}}{|\BMS|} \int_0^\infty e^{\delta_{\Gamma}r} \int_0^\pi \prod_{\substack{\gamma \in \Gamma / \Gamma_{\vect{w}} \\ \gamma \neq \Gamma_{\vect{w}}}} \left( 1- \chi_{\mathcal{E}(\gamma)}(r,\theta)\right) d\nu_{\vect{i}}^{\vect{w}}(\theta) dr.
    \end{equation}
\end{theorem}

Given $\gamma$ one can compute $\mathcal{E}(\gamma)$ explicitly, however the conformal density $\nu_{\vect{i}}$ is defined as the weak limit of a sequence of measures. Hence computing the gap distribution exactly will require more knowledge to get around this complication. When $\Gamma$ is a lattice, \eqref{gap dist explicit} can be written
\begin{equation}\label{gap dist explicit lattice}
  F(L) = \frac{\vartheta}{\vol_{\half^2}( \half^2/ \Gamma)} \int_0^\infty e^{r} \int_0^\pi \prod_{\substack{\gamma \in  \Gamma / \Gamma_{\vect{w}} \\\ \gamma \neq \Gamma_{\vect{w}}}} \left( 1- \chi_{\mathcal{E}(\gamma)}(r,\theta)\right) d\theta dr.
\end{equation}

To our knowledge, even in the lattice case, this is the first general explicit formula for the gap distribution. The gap distribution has been calculated explicitly for specific examples (notably \cite{RudnickZhang2017} who study the problem in certain circle packings). \eqref{gap dist explicit lattice} can be derived from \cite{MarklofVinogradov2018}, where the authors perform a similar calculation for the pair correlation.

Finally we turn to the gap distribution. That is, we define

\begin{equation}
  P_t(s) := \frac{1}{N_t} \sum_{j =1}^{N_t}\delta(s-s_j)
\end{equation}
where $\delta$ denotes a Dirac mass at the origin. By {\cite[Proof of Theorem 2]{Marklof2007}}, the existence of $F(L)$, and its derivative, imply the weak-$*$ limit of $P_t(s)$ as $t \to \infty$, which we denote $P(s)$. Moreover, by the same {\cite[Theorem 2]{Marklof2007}}, $P(L) =-F^\prime(L)$.

Now, given $\gamma$, $L$, and $\theta$ let

\begin{equation}
  e^{-r(\gamma,L,\theta)} : = \frac{\vartheta^{-1/\delta_{\Gamma}}L (\cosh(l(\gamma)) - \sinh(l(\gamma)))\cos(2(\theta+\theta(\gamma)))}{\sinh(l(\gamma)) \sin(2(\theta + \theta(\gamma)))},
\end{equation}
let $r_{\min}(L,\theta) = \min_{\gamma \in \Gamma/ \Gamma_{\vect{w}},  \gamma \neq \Gamma_{\vect{w}}} r(\gamma,L,\theta)$ and let $\gamma_{\max}(L,\theta)$ be the $\gamma$ maximizing that equation. In this case, recall that $P(L) = - F'(L)$, then

\begin{align}
  \begin{aligned}
    &P(L) = \frac{\vartheta e^{(1-\delta_{\Gamma})r_{\vect{w}}}}{|\BMS|}\int_0^{\pi} \mathbbm{1}(r_{\min}(L,\theta) < 0 ) e^{-\delta_{\Gamma}r_{\min}(L,\theta)}\\
    &\phantom{+++++++++++++}\prod_{\substack{\gamma \in  \Gamma / \Gamma_{\vect{w}} \\ \gamma_{\max}(L,\theta) \neq \gamma \neq \Gamma_{\vect{w}}}} \left( 1- \chi_{\mathcal{E}(\gamma)}(r_{\min}(L,\theta),\theta)\right) d\nu_{\vect{i}}^{\vect{w}}(\theta).
    \end{aligned}
\end{align}
The conditions on $\theta$ are now equivalent to

\begin{equation}
  I(\gamma) = \left\{\theta \in [0,\pi) :\frac{e^{-r_{\min}(L,\theta)}}{ \sinh(l(\gamma))\cos(\theta+\theta(\gamma)) - \cosh(l(\gamma))}<1, \quad r_{\min}(L,\theta) < 0 \right\}
\end{equation}
In which case

\begin{equation}
  P(L) = \frac{\vartheta e^{(1-\delta_{\Gamma})r_{\vect{w}}}}{|\BMS|}\int^{\pi}_{0} e^{-\delta_{\Gamma}r_{\min}(L,\theta)}  \prod_{\substack{\gamma \in \Gamma / \Gamma_{\vect{w}} \\ \gamma_{\max}(L,\theta) \neq \gamma \neq \Gamma_{\vect{w}}}} \chi_{I(\gamma)}(\theta) d\nu_{\vect{i}}^{\vect{w}}(\theta).
\end{equation}





\section*{Acknowledgements}

The author was supported by EPSRC Studentship EP/N509619/1 1793795. Furthermore the author is extremely grateful to Jens Marklof for his guidance throughout this project. Moreover to Sam Edwards, Hee Oh and Xin Zhang for insightful comments on an earlier draft and in particular, for Xin's comment regarding Winter's mixing theorem that allowed us to prove our main theorem in greater generality. Lastly the author is very grateful to the referee for their thorough and insightful comments.

\small
  \bibliographystyle{abbrv}
  \bibliography{biblio}

\end{document}